\documentclass[11pt,reqno,a4paper]{amsart}
\usepackage[dvips]{graphicx}
\usepackage{psfrag, enumerate, subfigure, multicol, hyperref}
\usepackage{mathrsfs}

\usepackage{amssymb, bbm}

\newtheorem{theorem}{Theorem}
\newtheorem{lemma}{Lemma}
\newtheorem{proposition}{Proposition}
\newtheorem{corollary}{Corollary}

\newtheorem{remark}{Remark}

\newtheorem{example}{Example}

\newtheorem{definition}{Definition}

 \newcommand{\reals}{{\mathbb{R}}}
 \newcommand{\naturals}{{\mathbb{N}}}
 \newcommand{\integers}{{\mathbb{Z}}}
 \newcommand{\C}[1]{\mathbf{C^{#1}}}
 \renewcommand{\L}[1]{\mathbf{L^{#1}}}
 \newcommand{\Lip}{\mathbf{Lip}}
 \newcommand{\lip}{\mathrm{Lip}}
 \newcommand{\Cc}[1]{\mathbf{C_c^{#1}}}
 \newcommand{\PC}{\mathbf{PC}}
 \renewcommand{\d}{\rm{d}}
 \newcommand{\iw}{\textrm{i}_{w}}
 \newcommand{\caratt}[1]{{\displaystyle\chi_{\strut{\textstyle #1}}}}
 \newcommand{\modulo}[1]{{\left|#1\right|}}
 \newcommand{\norma}[1]{{\left\|#1\right\|}}
 \newcommand{\BV}{\mathbf{BV}}
 \newcommand{\tv}{\mathrm{TV}}
 \newcommand{\sign}{\mathrm{sign}}
 \newcommand{\supp}{\mathrm{supp}}
 \newcommand{\Lloc}[1]{\mathbf{L^{#1}_{loc}}}

\begin{document}

\title{CROWD DYNAMICS AND CONSERVATION LAWS WITH NON--LOCAL CONSTRAINTS}

\author{Boris Andreianov}

\address{Laboratoire de Math\'ematiques,
Universit\'e de Franche-Comt\'e,\\
16 route de Gray,
25030 Besan\c{c}on Cedex,
France}
\email{boris.andreianov@univ-fcomte.fr}

\author{Carlotta Donadello}

\address{Laboratoire de Math\'ematiques,
Universit\'e de Franche-Comt\'e,\\
16 route de Gray,
25030 Besan\c{c}on Cedex,
France}
\email{carlotta.donadello@univ-fcomte.fr}

\author{Massimiliano D.~Rosini}

\address{
ICM, University of Warsaw\\
ul.~Prosta 69,
P.O. Box 00-838,
Warsaw, Poland}
\email{mrosini@icm.edu.pl}

\maketitle

{
\rightskip .85 cm
\leftskip .85 cm
\parindent 0 pt
\begin{footnotesize}

{\sc Abstract.} In this paper we model pedestrian flows evacuating a narrow corridor through an exit by a one--dimensional hyperbolic conservation law with a non--local constraint. Existence and stability results for the Cauchy problem with Lipschitz constraint are achieved by a procedure that combines the wave--front tracking algorithm with the operator splitting method. The Riemann problem with piecewise constant constraint is discussed in details, stressing the possible lack of uniqueness, self--similarity and $\Lloc1$--continuity. One explicit example of application is provided.



\end{footnotesize}

}

\section{Introduction}

The theory for constrained conservation laws was introduced by Colombo and Goatin in Ref.~\cite{ColomboGoatinConstraint}. Their results are of interest in many real--life applications, such as vehicular traffic, Refs.~\cite{CGRESAIM}, \cite{GaravelloGoatin}, pedestrian flows, Refs.~\cite{chalonsgoatinseguin}, \cite{ColomboRosini1}, telecommunications, supply--chains, etc.

For pedestrians, constraints are usually caused by a direct capacity reduction (door or obstacle) and are of fundamental importance in the calculation of evacuation times. The first macroscopic model for pedestrian evacuations able to reproduce the fall in the efficiency of an exit when a high density of pedestrians clogs it, was the CR~model proposed in Ref.~\cite{ColomboRosini1} and developed in Refs.~\cite{colombo2010macroscopic}, \cite{gakoto}, \cite{ColomboRosini2}, \cite{Rosini09}; see also Ref.~\cite{Rosinibook}. There, the maximal outflow allowed through the exit is assumed to be a piecewise constant function of the density at the exit, and takes two distinct values, one related to the ``standard'' case, when the density is less than an assigned threshold, and one related to the case with ``panic'', when the density is greater than the threshold. As a result, the fall in the efficiency of the exit has a non--realistic behavior since it is instantaneous when the panic reaches the exit. A more realistic model should reproduce a more gradual decay in the efficiency of the exit as the pedestrians accumulate close to it, see Refs.~\cite{Predtetschenski1971b}, \cite{Steffen20101902}. Moreover, according to the CR~model, once the efficiency of the exit falls down, it remains constant until the very last pedestrian is evacuated. On the contrary, in real life the efficiency of the exit gradually increases as the number of the remaining pedestrians to be evacuated becomes smaller and smaller.

To avoid these drawbacks of the CR~model, in this paper we generalize the results proved in Ref.~\cite{ColomboGoatinConstraint} and study the Cauchy problem for a one--dimensional hyperbolic conservation law with non--local constraint of the form
\begin{subequations}\label{eq:constrianed}
\begin{align}\label{eq:constrianed1}
    \partial_t\rho + \partial_xf(\rho) &= 0 & (t,x) &\in \reals_+\times\reals\\
    \label{eq:constrianed2}
    f\left(\rho(t,0\pm)\right) &\le p\left( \int_{\reals_-} w(x) ~ \rho(t,x) ~{\d} x\right) & t &\in \reals_+\\
    \label{eq:constrianed3}
    \rho(0,x) &=\rho_0(x) & x &\in \reals ~.
\end{align}
\end{subequations}
Above, $\rho = \rho(t,x) \in \left[0,R\right]$ is the (mean) density at time $t \in \reals_+$ of pedestrians moving along the corridor parameterized by the coordinate $x \in \reals_-$. Then, $R \in \reals_+$ is the maximal density, $f ~\colon~ [0,R] \to \reals$ is the pedestrian flow with pedestrians moving in the direction of increasing $x$, $p ~\colon~ \reals_+ \to \reals_+$ prescribes the maximal flow allowed through an exit placed in $x=0$ as a function of the weighted average density of pedestrians in a left neighborhood of the exit, $w  ~\colon~ \reals_- \to \reals_+$ is the weight function used for the average density and $\rho_0  ~\colon~ \reals \to [0,R]$ is the initial (mean) density. Finally, $\rho(t,0-)$ denotes the left measure theoretic trace along the constraint implicitly defined by
\begin{align*}
    \lim_{\varepsilon\downarrow0}\frac{1}{\varepsilon} \int_0^{+\infty}  \int_{-\varepsilon}^0 \modulo{\rho(t,x) - \rho(t,0-)} ~\phi(t,x) ~{\d} x ~{\d} t&=0
\end{align*}
for all $\phi \in \Cc\infty(\reals^2;\reals)$. The right measure theoretic trace, $\rho(t,0+)$, is defined analogously.

Observe that if $w$ is regular enough and $\lim_{x \to -\infty} w(x) = 0$, then the quantity
\begin{align}\label{eq:xi}
    \xi(t) = \int_{\reals_-} w(x) ~\rho(t,x) ~{\d}x
\end{align}
is the solution of the following Cauchy problem for an ordinary differential equation
\begin{align*}
    \dot\xi(t) &= \int_{\reals_-} \dot w(x) ~\left[ f\left(\rho(t,x)\right) - f\left(\rho(t,0-)\right) \right] ~{\d}x~,&
    \xi(0) &= \int_{\reals_-} w(x) ~\rho_0(x) ~{\d}x ~.
\end{align*}

In real life, when a very high density of pedestrians accumulate near the exit, the outgoing flow can be very small, but remains strictly positive. For this reason, the efficiency of the exit $p$ is assumed to be always strictly positive. We assume that the weight $w$ is an increasing function with compact support because the efficiency of the exit is more affected by the closest high densities, while it does not take into account ``far'' densities. In summary, we assume that:
\begin{enumerate}[\textbf{(P0)}]
  \item[\textbf{(F)}]  $f \in \Lip\left( [0,R]; \left[0, +\infty\right[ \right)$, $f(0) = 0 = f(R)$ and there exists $\bar\rho \in \left]0,R\right[$ such that $f'(\rho)~(\bar\rho-\rho)>0$ for a.e.~$\rho \in [0,R]$.
  \item[\textbf{(W)}] $w \in \L\infty(\reals_-;\reals_+)$ is an increasing map, $\norma{w}_{\L1(\reals_-;\reals_+)} = 1$ and there exists $\iw >0$ such that $w(x) = 0$ for any $x \le -\iw$.
  \item[\textbf{(P0)}] $p$ takes values in $\left]0,f(\bar\rho)\right]$ and is a non--increasing map.
\end{enumerate}

\begin{figure}[htpb]
 \centering
  \includegraphics[width=.75\textwidth]{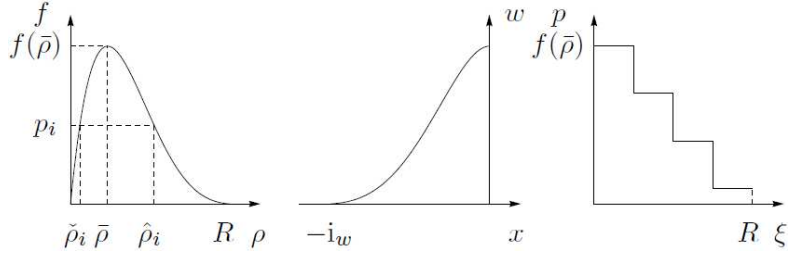}
  \caption{Examples of functions satisfying conditions~\textbf{(F)}, \textbf{(W)}, \textbf{(P0)} and \textbf{(P2)}.}
\label{fig:Korn}
\end{figure}

Observe that $f(\rho) < f(\bar\rho)$ for any $\rho \ne \bar\rho$ and $\xi(t) \in \left[0, R\right]$, see Fig.~\ref{fig:Korn}. In the present work we do not take into account the presence of a panic regime since in~\textbf{(F)} we assume that the fundamental diagram $[\rho \mapsto \left(\rho, f(\rho)\right)]$ is bell--shaped. Indeed, the CR~model introduces a flux that results from the juxtaposition of two bell--shaped sub--fluxes corresponding to the two regimes quiet--panic and, therefore, does not satisfy the condition~\textbf{(F)}. The latter assumption~\textbf{(P0)} is the minimal requirement for~\eqref{eq:constrianed} to be meaningful in the sense of distributions, see Definition~\ref{def:entropysol}. While existence for the Riemann problem is proved for piecewise constant $p$, see~\textbf{(P2)} in Sec.~\ref{sec:Riemann}, we strengthen the assumption on $p$ to a Lipschitz continuity hypothesis when dealing with the Cauchy problem, see~\textbf{(P1)} and Theorem~\ref{thm:1} in Sec~\ref{sec:Cauchy}.

We  give the definition of solution for problem with nonlocal constraint~\eqref{eq:constrianed} by extending the definition of entropy weak solution for a constrained Cauchy problem of the form
\begin{subequations}\label{eq:constrianedLOCAL}
\begin{align}\label{eq:constrianedLOCAL1}
    \partial_t\rho + \partial_xf(\rho) &= 0 & (t,x) &\in \reals_+\times\reals\\
    \label{eq:constrianedLOCAL2}
    f\left(\rho(t,0\pm)\right) &\le q\left(t\right) & t &\in \reals_+\\
    \label{eq:constrianedLOCAL3}
    \rho(0,x) &=\rho_0(x) & x &\in \reals ~.
\end{align}
\end{subequations}
\begin{definition}\label{def:entropysol}
    Assume conditions~\textbf{(F)}, \textbf{(W)}, \textbf{(P0)}. A map $\rho \in \L\infty( \reals_+\times\reals;[0,R]) \cap \C0(\reals_+; \Lloc1(\reals;[0,R]))$ is an entropy weak solution to~\eqref{eq:constrianed} if there exists $q \in \L\infty( \reals_+;[0, f(\bar\rho)])$ such that the following conditions hold:
    \begin{enumerate}
    \item\label{cond:1b} For every test function $\phi \in \Cc\infty(\reals^2; \reals_+)$ and for every $k \in [0,R]$
    \begin{subequations}\label{eq:entropysol}
    \begin{align}\label{eq:entropysol1}
        &\int_{\reals_+} \int_\reals \left[ \modulo{\rho-k} \partial_t\phi + \sign(\rho-k) \left(f(\rho) - f(k)\right) \partial_x\phi \right] ~{\d} x ~{\d} t\\ \label{eq:entropysol2}
        &+ 2 \int_{\reals_+} \left[1 - \dfrac{q\left( t\right)}{f(\bar\rho)}   \right] f(k) ~\phi(t,0) ~{\d} t\\ \label{eq:entropysol3}
        &+ \int_\reals \modulo{\rho_0(x) - k} ~\phi(0,x) ~{\d} x
        \ge 0~,
    \end{align}
    and
    \begin{align}\label{eq:entropysol4}
        &f\left(\rho(t, 0\pm)\right) \le q\left(t\right) \hbox{ for a.e.~}t \in \reals_+~.
    \end{align}
    \end{subequations}
    \item In addition $q$ is linked to $\rho$ by the relation
    \begin{align}\label{eq:entropysol5}
        q(t) = p\left( \int_{\reals_-} w(x) ~ \rho(t,x) ~{\d} x\right) \hbox{ for a.e.~}t \in \reals_+~.
    \end{align}
    \end{enumerate}
\end{definition}
If $q$ is given \textit{a priori}, then~\eqref{eq:entropysol} is the definition of entropy weak solution
to problem~\eqref{eq:constrianedLOCAL}. This item is precisely the Definition~2.1 in Ref.~\cite{scontrainte}, which is a minor generalization of the original Definition~3.2 introduced in Ref.~\cite{ColomboGoatinConstraint}.
We refer to  Proposition~2.6 in Ref.~\cite{scontrainte} for a series of equivalent formulations of conditions~\eqref{eq:entropysol}.

The lines~\eqref{eq:entropysol1} and~\eqref{eq:entropysol3} originate from the classical Kru\v zkov Definition~1 in Ref.~\cite{Kruzkov}, in the case of the Cauchy problem with no constraints. Lines~\eqref{eq:entropysol2} and~\eqref{eq:entropysol4} account for the constraint. Let us stress that both left and right traces at $x=0$ of an entropy weak solution exist (see, e.g., Theorem~2.2 in Ref.~\cite{scontrainte} which is a reformulation of the results of Refs.~\cite{Panov}, \cite{Vasseur}).

Our main result is well-posedness for the nonlocal problem~\eqref{eq:constrianed}, see Theorem~\ref{thm:1}. We show that under the Lipschitz continuity assumption on $p$ there exists a semigroup $(\mathcal{S}_t)_{t>0}$ on $\L\infty(\reals;[0,R])$
such that $\rho(t,\cdot)=\mathcal{S}_t(\rho_0)$ is the unique solution of~\eqref{eq:constrianed}, and depends continuously on $t$ and $\rho_0$ with respect to the $\Lloc1$--distance.

The uniqueness result for~\eqref{eq:constrianed} is a consequence of a stability estimate for the problem with local constraint~\eqref{eq:constrianedLOCAL} with respect to the $\Lloc1$--distance, of the relation~\eqref{eq:entropysol5} and of the Gronwall inequality.

The existence result for~\eqref{eq:constrianed} is achieved through an operator splitting method, Refs.~\cite{ColomboCorli5}, \cite{ColomboCorliRosini}, \cite{ColomboRosini}, \cite{ColomboRosini3}, \cite{DafermosBook}, coupled with the wave--front tracking algorithm, Refs.~\cite{ColomboGoatinConstraint}, \cite{DafermosWFT}. This procedure is chosen for two reasons. First, wave--front tracking schemes are able to operate also in the case with panic, when nonclassical shocks away from the constraint have to be taken into account, see Ref.~\cite{chalonsgoatinseguin}. Second, the operator splitting procedure allows us to approximate our problem with a problem of type~\eqref{eq:constrianedLOCAL}, namely with a ``frozen'' constraint. This greatly simplify our work because it avoids the difficulties coming from the Riemann solver for the nonlocally constrained problem. Indeed, we underline the fact that, differently from the constrained Cauchy problems studied in Refs.~\cite{scontrainte}, \cite{ColomboGoatinConstraint}, the maximal flow at the constraint for~\eqref{eq:constrianed} depends on the solution itself and, in general, it is an unknown variable of the problem. As soon as $p$ is discretized, \textit{i.e.}~$p$ is approximated by piecewise constant functions, the solution of the corresponding nonlocally constrained Riemann problem may fail to be unique, $\Lloc1$--continuous, consistent and self--similar, as we will show in Sec.~\ref{sec:Riemann}. Furthermore, using wave--front tracking for a nonlocal problem is quite delicate because one cannot merely juxtapose local solutions of Riemann problems, see Remark~\ref{rem:catastrofe}.

The use of wave--front tracking approximation requires the $\BV$ functional setting. As already observed in Ref.~\cite{ColomboGoatinConstraint}, the constraint may cause sharp increases in the total variation $\tv(\rho)$ of the solution. To overcome this difficulty, as in Refs.~\cite{CocliteRisebro2005}, \cite{ColomboGoatinConstraint}, \cite{CGRESAIM}, \cite{Temple82}, we rather estimate the total variation of $\Psi\circ\rho$, where
\begin{equation}\label{eq:psi}
    \Psi(\rho) = \sign(\rho-\bar\rho) ~ \left(f(\bar\rho)-f(\rho)\right)=\int_{\bar \rho}^\rho \modulo{\dot f(r)} ~{\d}r~.
\end{equation}
We stress that $\Psi$ is one--to--one, but possibly singular at $\rho = \bar\rho$. Indeed, if $\rho$ is in $\BV$, then also $\Psi\circ\rho$ is in $\BV$, while the reverse implication does not hold true. Therefore we introduce the set
\begin{equation}\label{eq:domain}
    \mathcal{D} = \left\{ \rho \in \L1 \left(\reals ; [0,R]\right) ~\colon~ \Psi(\rho) \in \BV(\reals;\reals) \right\}~.
\end{equation}
In the first step of our construction, we exploit the finite speed of propagation property for the conservation law~\eqref{eq:constrianed1} to define $\mathcal{S}_t$ on the domain $\mathcal{D}$ by coupling the operator splitting method with the wave--front tracking algorithm, see Sec~\ref{sec:3}. Then we extend $\mathcal{S}_t$ to $\L\infty$ by a density argument, see the second part of the proof of Theorem~\ref{thm:1}.

The paper is organized as follows. Sec.~\ref{sec:Cauchy} and~\ref{sec:3} are devoted to the constrained Cauchy problem. In Sec.~\ref{sec:Riemann} we study~\eqref{eq:constrianed} with a Riemann initial datum. In Sec.~\ref{sec:example} we apply the model~\eqref{eq:constrianed} to describe the evacuation of a corridor through an exit placed at $x=0$. All the technical proofs are in Sec.~\ref{sec:tech}. Conclusions and perspectives are outlined in Sec.~\ref{sec:conlusion}.

\section{The Cauchy problem with nonlocal constraint}\label{sec:Cauchy}

In this section we consider the Cauchy problem~\eqref{eq:constrianed} under the hypotheses~\textbf{(F)}, \textbf{(W)} and the following assumption on $p$:
\begin{enumerate}
\item[\textbf{(P1)}] $p$ belongs to $\Lip \left( \left[0,R\right]; \left]0,f(\bar\rho)\right] \right)$ and it is a non-increasing map.
\end{enumerate}
Let us start with the basic properties of entropy weak solutions to~\eqref{eq:constrianed}.
\begin{proposition}\label{prop:ws}
    Let $[t \mapsto \rho(t)]$ be an entropy weak solution of~\eqref{eq:constrianed} in the sense of Definition~\ref{def:entropysol}. Then
    \begin{enumerate}[(1)]
      \item It is also a weak solution of~\eqref{eq:constrianed1}, \eqref{eq:constrianed3}.
      \item Any discontinuity satisfies the Rankine--Hugoniot jump condition.
      \item Any discontinuity away from the constraint is classical, \textit{i.e.}~satisfies the Lax entropy inequalities.
      \item Nonclassical discontinuities, see Refs.~\cite{LeflochBook}, \cite{Rosinibook}, may occur only at the constraint location $x=0$, and in this case the flow at $x=0$ is the maximal flow allowed by the constraint. Namely, if the solution contains a nonclassical discontinuity for all times $t \in I$, $I$ open in $\reals_+$, then for a.e.~$t$ in $I$
          \begin{align}\label{eq:ws}
            f\left( \rho(t, 0-) \right) = f\left( \rho(t, 0+) \right) = p\left( \int_{\reals_-} w(x) ~ \rho(t,x) ~{\d} x\right) ~.
          \end{align}
    \end{enumerate}
\end{proposition}
\begin{proof}
    By taking $k=0$, then $k=R$, in~\eqref{eq:entropysol}, we deduce that any entropy weak solution to~\eqref{eq:constrianed} is also a weak solution to~\eqref{eq:constrianed1}, \eqref{eq:constrianed3}. As a consequence, $\rho$ satisfies the Rankine--Hugoniot jump condition and, in particular, $f\left(\rho(t, 0-)\right) = f\left(\rho(t, 0+)\right)$. By taking in~\eqref{eq:entropysol} a test function with support in $\reals_+ \times \reals_-$, then in $\reals_+ \times \reals_+$, we see that $\rho$ is also a classical Kru\v zkov solution to~\eqref{eq:constrianed1}, \eqref{eq:constrianed3} in $\reals_+ \times \reals_\pm$ and therefore the jumps in $\rho$ located at $x\neq0$ satisfy the Lax entropy inequalities. Finally, we prove property~\eqref{eq:ws}. This property was observed at the level of problem~\eqref{eq:constrianedLOCAL}, see in particular the description of the ``germ'' $\mathcal{G}_F$ in Ref.~\cite{scontrainte}, but it was not explicitly stated in the works Refs.~\cite{scontrainte}, \cite{ColomboGoatinConstraint} devoted to problem~\eqref{eq:constrianedLOCAL}. For the sake of completeness, we give an explicit proof of property~\eqref{eq:ws} for problem~\eqref{eq:entropysol}. Consider the test function
    \begin{align*}
        \phi(t,x) &= \left[\int_{\modulo{x}-\varepsilon}^{+\infty} \delta_\varepsilon(z) ~{\d}z \right] \psi(t) ~,
    \end{align*}
    where $\psi\in\Cc\infty(\reals;\reals_+)$ is such that $\psi(0)=0$, while $\delta_\varepsilon$ is a smooth approximation of the Dirac mass centered at $0+$, $\delta^D_{0+}$, namely
    \begin{align}\label{eq:delta}
        \delta_\varepsilon \in \Cc\infty(\reals; \reals_+),
        ~\varepsilon \in \reals_+,
        ~\supp(\delta_{\varepsilon}) \subseteq[0,\varepsilon],
        ~\norma{\delta_\varepsilon}_{\L1(\reals;\reals)} = 1,
        ~\delta_\varepsilon \to \delta^D_{0+}.
    \end{align}
    Observe that as $\varepsilon$ goes to zero
    \begin{align*}
        \phi(0,x) &\equiv0 \to 0 ~,
        && \partial_t\phi(t,x) = \left[\int_{\modulo{x}-\varepsilon}^{+\infty} \delta_\varepsilon(z) ~{\d}z \right] \dot\psi(t)\to 0 ~,\\
        \phi(t,0) &=\psi(t) \to \psi(t) ~,
        && \caratt{\reals_\pm}(x) ~\partial_x \phi(t,x) \to \mp ~\delta^D_{0\pm}(\modulo{x}) ~\psi(t) ~.
    \end{align*}
    Then, if we take $k=\bar\rho$ and $\phi$ as test function in~\eqref{eq:entropysol}, we obtain as $\varepsilon$ goes to zero
    \begin{align*}
        &\int_{\reals_+} \left[ \Psi\left(\rho(t,0+)\right) - \Psi\left(\rho(t,0-)\right) \right]
        \psi(t) ~{\d}t
        +2\int_{\reals_+} \left[
        f(\bar\rho) - p\left(\xi(t)\right)
        \right]
        ~\psi(t) ~{\d}t \ge 0 ~,
    \end{align*}
    where $\xi$ is defined by~\eqref{eq:xi}. For the arbitrariness of $\psi$, we have for a.e.~$t>0$
    \begin{align*}
        &\Psi\left(\rho(t,0+)\right) - \Psi\left(\rho(t,0-)\right)
        +2 \left[f(\bar\rho) - p\left(\xi(t)\right)\right] \ge0 ~.
    \end{align*}
    Therefore, if for $t \in I$ the solution has a nonclassical discontinuity at the constraint location $x=0$, then by the assumption~\textbf{(F)} and the Rankine--Hugoniot jump condition, $\rho(t, 0+) < \bar\rho < \rho(t, 0-)$ and $p\left(\xi(t)\right) \le f\left(\rho(t,0\pm)\right)$ for a.e.~$t \in I$. Finally, by the condition~\eqref{eq:entropysol4} of Definition~\ref{def:entropysol}, it has to be $p\left(\xi(t)\right) = f\left(\rho(t,0\pm)\right)$ for a.e.~$t \in I$.
\end{proof}

The following theorem on existence, uniqueness and stability of entropy weak solutions of the constrained Cauchy problem~\eqref{eq:constrianed} is the main result of this paper.
\begin{theorem}\label{thm:1}
    Let~\textbf{(F)}, \textbf{(W)}, \textbf{(P1)} hold. Then
     \begin{itemize}
       \item[(i)] For any initial datum $\rho_0 \in \L\infty(\reals;[0,R])$, the Cauchy problem~\eqref{eq:constrianed} admits a unique entropy weak solution $\rho$
           in the sense of Definition~\ref{def:entropysol}. Moreover, if $\tilde\rho = \tilde\rho(t,x)$ is the entropy weak solution corresponding to the initial datum $\tilde\rho_0 \in \L\infty(\reals;[0,R])$, then for all~$T>0$ and $L>\iw$ there holds
           \begin{align}\label{eq:lipdepen-localized}
                \norma{\rho(T) - \tilde\rho(T)}_{\L1([-L,L];\reals)} &\le e^{CT}\norma{\rho_0 - \tilde\rho_0}_{\L1([-(L+MT),(L+MT)];\reals)},
           \end{align}
       where $M=\lip(f)$ and $C=2 \lip(p) \norma{w}_{\L\infty(\reals_-;\reals)}$.
       \item[(ii)] If $\rho_0$ belongs to $\mathcal{D}$, defined as in~\eqref{eq:domain}, then the unique entropy weak solution of problem~\eqref{eq:constrianed} verifies $\rho(t,\cdot)\in \mathcal{D}$ for a.e. $t>0$, and it satisfies
           \begin{align}\label{eq:BVdepen}
                &\tv\left( \Psi\left( \rho(t) \right) \right)
                \le C_t = \tv\left( \Psi\left( \rho_0 \right) \right) + 4f(\bar\rho) + C ~t~,
           \end{align}
           moreover, for a.e.~$t,s$ in $]0,T[$ we have
           \begin{align}\label{eq:BVdepen2}
                &\norma{\Psi\left(\rho(t,\cdot)\right) - \Psi\left(\rho(s,\cdot)\right)}_{\L1(\reals;\reals)}
                \le \modulo{t-s} ~\lip(\Psi) ~C_T ~.
           \end{align}
     \end{itemize}
\end{theorem}

\begin{proof}
The proof consists of three parts, the longest one being postponed to Sec.~\ref{sec:3}.

\vspace{.3cm}
\noindent{\bf{Uniqueness and stability}}
\vspace{.1cm}

\noindent Conditions~\eqref{eq:entropysol} of Definition~\ref{def:entropysol} ensure that for all $t \in [0,T]$ we can apply the stability estimate in Proposition~2.10 in Ref.~\cite{scontrainte}. More specifically, if $\rho$ and $\tilde\rho$ are solutions of~\eqref{eq:constrianed} corresponding to the constraints $q$ and $\tilde q$, and the initial conditions $\rho_0$ and $\tilde\rho_0$, respectively, then
\begin{align*}
    \norma{ \rho(t) - \tilde\rho(t) }_{\L1([-L,L];\reals)}
    \le \norma{\rho_0-\tilde\rho_0}_{\L1(\{\modulo{x}\le L+Mt\};\reals)}
    +2\int_0^t \modulo{q(s) - \tilde q(s)} ~{\d}s .
\end{align*}
By the explicit expression of the constraints $q$ and $\tilde q$, see~\eqref{eq:entropysol5}, we have
\begin{align*}
    \int_0^t \modulo{q(s) - \tilde q(s)} ~{\d}s
    = \int_0^t \modulo{p \left(\int_{\reals_-} w(x) ~\rho(s,x) ~{\d}x \right) - p \left(\int_{\reals_-} w(x) ~\tilde\rho(s,x) ~{\d}x \right)} ~{\d}s ,
\end{align*}
and this quantity is bounded by $\lip(p) ~w(0-) \int_0^t \norma{ \rho(s) - \tilde\rho(s) }_{\L1([-\iw,0];\reals)} {\d}s$ by the Lipschitz continuity of $p$ and H\"older inequality. Note the inclusions $[-L,L] \supseteq [-\iw,0]$ and $\{\modulo{x}\le L+MT\} \supseteq \{\modulo{x}\le L+Mt\}$. We complete the proof by applying Gronwall's inequality, see for instance Ref.~\cite{Rosinibook}.

\vspace{.3cm}
\noindent{\bf{Existence in $\mathcal{D}$}}
\vspace{.1cm}

\noindent The existence problem in the $\mathcal{D}$--framework will be addressed in Sec.~\ref{sec:3}, see Proposition~\ref{prop:Dexistence}. With this result in hand, existence in $\L\infty$ follows by the density argument we develop below.

\vspace{.3cm}
\noindent{\bf{Existence in $\L\infty$}}
\vspace{.1cm}

\noindent Let $\rho_0$ be in $\L\infty(\reals;[0,R])$. By the standard diagonal procedure argument, it is enough to prove existence on an arbitrary time interval $[0,T]$ in $\reals_+$. Introduce a sequence $\Psi_0^n$ in $\BV_{loc}(\reals;\reals)$ which converges pointwise a.e.~to $\Psi_0 = \Psi(\rho_0)$. Set $\rho_0^n = \Psi^{-1}(\Psi_0^n)$. For any $L \in \naturals$ sufficiently large, set $\rho_0^{n,L} = \rho_0^n ~\chi_{\{\modulo{x}\le L+MT\}}$. We have that $\rho_0^{n,L}$ belongs to $\mathcal{D}$ and $\rho_0^{n,L}$ converges to $\rho_0$ in $\Lloc1(\reals;[0,R])$ as $L$ and $n$ go to infinity. Let $\rho^{n,L}$ be the corresponding entropy weak solution of~\eqref{eq:constrianed} constructed in Proposition~\ref{prop:Dexistence}. Then, for any $L>1+\iw$
\begin{enumerate}[(A)]
  \item for all $t \in [0,T]$ we have that $\norma{\rho^{n,L}(t) - \rho^{m,L}(t)}_{\L1([-L,L];\reals)}$ goes to zero as $m$ and $n$ go to infinity;
  \item if $L' > L$, then $\rho^{n,L} \equiv \rho^{n,L'}$ on $[0,T] \times [-L,L]$.
\end{enumerate}
Properties~(A) and~(B) follow by~\eqref{eq:lipdepen-localized}. Then, by taking $L(x) = \lfloor\modulo{x}\rfloor +1$, property~(A) ensures that we can introduce the function $\rho(t,x) = \lim_{n\to+\infty}\rho^{n,L(x)}(t,x)$. By~(B) we also have that $\rho(t,x) = \lim_{n\to+\infty}\rho^{n,L'}(t,x)$ for any $L' > L(x)$.

We prove now that $[t \mapsto \rho(t)]$ is an entropy weak solution to~\eqref{eq:constrianed} with initial datum $\rho_0$. For any compact set $K \subset\reals$, take $L$ such that $[-L+1,L-1] \supseteq (K \cup [-\iw,0])$. Then $\rho^{n,L}$ converges to $\rho$ in $\L1([0,T]\times K;[0,R])$ and consequently for a.e.~$t \in [0,T]$, $q^{n,L}(t) = p\left(\int_{-\iw}^0 w(x) ~\rho^{n,L}(t,x) ~{\d}x\right)$ converges to $q(t) = p\left(\int_{-\iw}^0 w(x) ~\rho(t,x) ~{\d}x\right)$ in $\L1([0,T];\reals)$. This is enough to ensure that for all test functions $\phi$ supported in $[0,T] \times K$, the function $\rho$ satisfies~\eqref{eq:entropysol1}--\eqref{eq:entropysol3} and that~\eqref{eq:entropysol5} holds. In particular the Rankine--Hugoniot condition is satisfied, therefore using Lemma~\ref{lem:star}, we have that $f^n\left(\rho^{n,L}(\cdot,0-)\right)$ converges weakly to $f\left(\rho(\cdot,0-)\right)$ in $\L1([0,T];\reals)$ and $f\left(\rho(\cdot,0-)\right) = f\left(\rho(\cdot,0+)\right)$. Therefore, also~\eqref{eq:entropysol4} holds true.
\end{proof}

\section{Wave--front tracking and operator splitting methods}\label{sec:3}

In this section we construct solutions for initial data in $\mathcal{D}$ and we prove that $\mathcal{D}$ is an invariant domain for the semigroup $\mathcal{S}$.
\begin{proposition}\label{prop:Dexistence}
  For any initial datum $\rho_0$ in $\mathcal{D}$, there exists a unique entropy weak solution of problem~\eqref{eq:constrianed}, $[t \mapsto \rho(t)]$ and $\rho(t)$ belongs to $\mathcal{D}$ for all $t>0$. Moreover, estimates~\eqref{eq:BVdepen} and~\eqref{eq:BVdepen2} are satisfied.
\end{proposition}
The solution $[t \mapsto \rho(t)]$ is the limit (along a subsequence) of a sequence obtained by combining the wave--front tracking algorithm and the operator splitting method. In the following subsections we describe the construction in full details.

\subsection{Approximation of flux and efficiency functions}\label{sec:technicalCauchy2}

Fix $h,n \in \naturals$ sufficiently large with $n \gg h$. Introduce the mesh
\begin{align*}
    \mathcal{M}^n = f^{-1} \left(2^{-n} f(\bar\rho) \naturals \cap \left[0,  f(\bar\rho)\right] \right)
\end{align*}
and the set
\begin{align*}
    \mathcal{D}^n &= \mathcal{D} \cap \PC\left(\reals; \mathcal{M}^n \right) ~,
\end{align*}
where $\PC\left(\reals; \mathcal{M}^n \right)$ is the set of piecewise constant functions defined on $\reals$, taking values in $\mathcal{M}^n$ and with a finite number of jumps. Approximate the flux $f$ with a piecewise linear, continuous flux $f^n ~\colon~ [0,R] \to [0,f(\bar\rho)]$, whose derivative exists in $[0,R] \setminus \mathcal{M}^n$ and such that $f^n$ coincides with $f$ on $\mathcal{M}^n$, see Fig.~\ref{fig:ph}, left. Clearly, $f^n$ satisfies condition~\textbf{(F)}. Consider $p^{-1}\left( f(\mathcal{M}^h) \cap p([0,R]) \right) = \{ \tilde\xi^h_0, \ldots, \tilde\xi^h_{m_h+2} \}$, with $0 \le \tilde\xi_{0}^h < \tilde\xi_{1}^h < \ldots < \tilde\xi_{m_h+2}^h \le R$, and observe that
\begin{align}\label{eq:LEinaudi0}
    \left(\tilde\xi_{i+1}^h - \tilde\xi_{i}^h\right) \lip(p) \ge p\left(\tilde\xi_{i+1}^h\right) - p\left(\tilde\xi_{i}^h\right) = 2^{-h} f(\bar\rho) ~.
\end{align}
\begin{figure}[htpb]
      \centering
        \includegraphics[width=.9\textwidth]{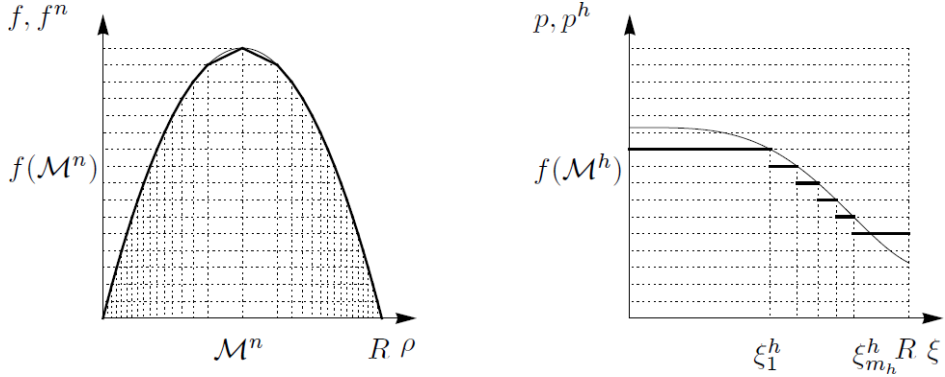}
      \caption{Left: : In thin line $f$ and in thick line the approximation $f^n$. Right: In thin line $p$ and in thick line the approximation $p^h$.}
\label{fig:ph}
\end{figure}
Approximate $p$ with the function $p^h \in \PC\left( [0,R]; f(\mathcal{M}^h) \right)$ defined as follows:
\begin{subequations}\label{eq:LEinaudi}
\begin{align}\label{eq:LEinaudi1}
    p^h(\xi) &= \sum_{i=0}^{m_h-1} p^h_i ~\caratt{\left[\xi_{i}^h, \xi_{i+1}^h\right[}(\xi) +
          p^h_{m_h} ~\caratt{\left[\xi_{m_h}^h, R\right]}(\xi)~,
\end{align}
where
\begin{align}\label{eq:LEinaudi2}
    0 &= \xi_{0}^h < \xi_{1}^h = \tilde\xi_{1}^h < \ldots < \xi_{m_h}^h = \tilde\xi_{m_h}^h < \xi_{m_h+1}^h =R ~,\\
    \label{eq:LEinaudi3}
    p^h_i &= p(\xi_{i+1}^h),~ i=0,\ldots,m_h-1,
    \hbox{ and }
    p^h_{m_h} = p(\tilde\xi_{m_h+1}^h) ~,
\end{align}
\end{subequations}
see Fig.~\ref{fig:ph}, right. Since $h<n$, we have that $f(\mathcal{M}^h) \subset f(\mathcal{M}^n)$, $p^h([0,R]) \subseteq f(\mathcal{M}^n) \setminus \{0\}$,
\begin{align}\label{eq:Looestimate4ph}
    \norma{p-p^h}_{\L\infty([0,R]; \reals)} \le 2^{1-h} f(\bar\rho)
\end{align}
and by~\eqref{eq:LEinaudi0} and~\eqref{eq:LEinaudi}
\begin{align}\label{eq:CFL0}
    p^h_i - p^h_{i+1} &= 2^{-h} f(\bar\rho) ~,
    \\ \label{eq:CFL1}
    \inf_{i=0,\ldots,m_h} \left(\xi_{i+1}^h - \xi_{i}^h\right) &\ge \frac{2^{-h} f(\bar\rho)}{\lip(p)} ~.
\end{align}

\subsection{The algorithm}\label{sec:Thealgorithm}

Now we can start with the construction of an approximating solution $[t \mapsto \rho^{n,h}(t)]$ to~\eqref{eq:constrianed}. As a first step we associate to any fractional time interval of the form $\left[\ell\Delta t_h, (\ell+1)\Delta t_h\right[$, $\Delta t_h>0$, $\ell \in \naturals$, a constrained Cauchy problem of the form~\eqref{eq:constrianedLOCAL} with constant constraint. Then the wave--front tracking algorithm gives us the corresponding exact solution $[t \mapsto \rho^{n,h}_{\ell+1}(t)]$. Finally, $\rho^{n,h}$ is obtained by gluing together $\rho^{n,h}_{\ell+1}$, $\ell \in \naturals$. The existence of a limit for $\rho^{n,h}$ as $n$ and $h$ go to infinity is ensured by the choice
\begin{align}\label{eq:CFL2}
    \Delta t_h &= \frac{1}{2^{h+1} w(0-) \lip(p)}~,
\end{align}
which will be motivated in the proof of Lemma~\ref{lem:capraEcavoli} by a sort of CFL~condition. Roughly speaking, this condition is needed to bound the possible jump in the value of the constraint due to the update at each fractional time $(\ell+1)\Delta t_h$, $\ell \in \naturals$.

Approximate $\rho_0$ with a piecewise constant function $\rho_0^n ~\colon~ \reals \to [0,R]$ that coincides with $\rho_0$ on $\mathcal{M}^n$ and such that $\norma{\rho_0^n}_{\L1(\reals;\reals)} \le \norma{\rho_0}_{\L1(\reals;\reals)}$ and $\tv\left( \Psi\left(\rho_0^n\right) \right) \le \tv\left( \Psi\left(\rho_0\right) \right)$. Clearly, $\rho_0^n$ belongs to $\mathcal{D}^n$. First consider the approximating constrained Cauchy problem
\begin{align*}
    \partial_t\rho + \partial_xf^n(\rho) &= 0 & (t,x) &\in \left]0, \Delta t_h\right[\times\reals\\
    f^n\left(\rho(t,0\pm)\right) &\le p^h\left( \Xi^{n}_{0} \right) & t &\in \left]0, \Delta t_h\right]\\
    \rho(0,x) &=\rho_0^n(x) & x &\in \reals ~,
\end{align*}
where
\begin{align*}
    \Xi^{n}_{0}
    &= \int_{\reals_-} w(x) ~\rho_0^n(x) ~{\d} x ~.
\end{align*}
The unique exact solution $[ t \mapsto \rho^{n,h}_1(t) ]$ for the above problem is obtained by piecing together the solutions to the Riemann problems at points where $\rho_0^n$ is discontinuous or where interactions take place, namely where two or more waves intersect, or one or more waves reach $x=0$. For the definition of solution of the Riemann problem with a piecewise linear, continuous flux away from the constraint, we refer to Sec.~6.1 in Ref.~\cite{BressanBook} or to Sec.~5.2 in Ref.~\cite{Rosinibook}. The definition of solution to the constrained Riemann problem along $x=0$ follows by the obvious adaptation of Definition~2.2 in Ref.~\cite{ColomboGoatinConstraint} to the case with a piecewise linear continuous flux, see also Sec.~6.3 in Ref.~\cite{Rosinibook}. The results of Theorem~3.4 in Ref.~\cite{ColomboGoatinConstraint} can be easily generalized to the case with piecewise linear continuous flux and, therefore, we can define
\begin{align*}
    \rho^{n,h}(t,x) & = \rho^{n,h}_1(t,x) & \hbox{ for }(t,x) \in \left]0, \Delta t_h\right] \times \reals~.
\end{align*}
We can assume that no interaction occurs at time $t  = \Delta t_h$, see assumption~\textbf{H2} below. Then the approximate solution is prolonged beyond $t = \Delta t_h$ by taking
\begin{align*}
    \rho^{n,h}(t,x) & = \rho^{n,h}_2(t-\Delta t_h,x) & \hbox{ for }(t,x) \in \left]\Delta t_h, 2\Delta t_h\right] \times \reals~,
\end{align*}
where $[ t \mapsto \rho^{n,h}_2(t) ]$ is the exact solution of the constrained Cauchy problem
\begin{align*}
    \partial_t\rho + \partial_xf^n(\rho) &= 0 & (t,x) &\in \left]0, \Delta t_h\right[\times\reals\\
    f^n\left(\rho(t,0\pm)\right) &\le p^h\left( {\Xi^{n,h}_{1}} \right) & t &\in \left]0, \Delta t_h\right]\\
    \rho(0,x) &=\rho_1^{n,h}(\Delta t_h,x) & x &\in \reals ~,
\end{align*}
with
\begin{align*}
    {\Xi^{n,h}_{1}} =
    \int_{\reals_-} w(x) ~\rho^{n,h}_{1}(\Delta t_h,x) ~{\d} x ~.
\end{align*}

We repeat this procedure at each fractional step and, once we get $[ t \mapsto \rho^{n,h}_\ell(t) ]$, we construct $[ t \mapsto \rho^{n,h}_{\ell+1}(t) ]$ by solving a constrained Cauchy problem of the form
\begin{subequations}\label{eq:constrianedapprk}
\begin{align}\label{eq:constrianedapprk1}
    \partial_t\rho + \partial_xf^n(\rho) &= 0 & (t,x) &\in \left]0, \Delta t_h\right[\times\reals\\
    \label{eq:constrianedapprk2}
    f^n\left(\rho(t,0\pm)\right) &\le p^h\left( {\Xi^{n,h}_{\ell}} \right) & t &\in \left]0, \Delta t_h\right]\\
    \label{eq:constrianedapprk3}
    \rho(0,x) &=\rho_\ell^{n,h}(\Delta t_h,x) & x &\in \reals ~,
\end{align}
where
\begin{align}
    \label{eq:constrianedapprk4}
    {\Xi^{n,h}_{\ell}} &=
    \int_{\reals_-} w(x) ~\rho^{n,h}_{\ell}(\Delta t_h,x) ~{\d} x ~.
\end{align}
\end{subequations}
We stress that the solution to~\eqref{eq:constrianedapprk} is unique and that the efficiency at the exit may change at each time $t \in \Delta t_h\naturals$ and only there.

To simplify the wave--front tracking algorithm, see Remark~7.1 in Ref.~\cite{BressanBook}, it is standard to remark that, without loss of generality, one can assume that:
\begin{enumerate}[H1]
  \item[\textbf{H1}] At any interaction either exactly two waves interact, or a single wave reaches the constraint $x=0$.
  \item[\textbf{H2}] No interaction occurs at time $t \in \Delta t_h \naturals$.
\end{enumerate}
In this way we construct
\begin{align}\label{eq:appreff}
    \Xi^{n,h}(t) &= \sum_{\ell\in\naturals} {\Xi^{n,h}_ \ell} ~\caratt{\left[ \ell\Delta t_h, (\ell+1) \Delta t_h \right[} (t)
\end{align}
and an approximate solution of the Cauchy problem~\eqref{eq:constrianed}
\begin{align}\label{eq:apprsol}
    \rho^{n,h}(t,x) &= \sum_{\ell\in\naturals} \rho^{n,h}_{\ell+1}(t-\ell\Delta t_h,x) ~\caratt{\left]\ell\Delta t_h, (\ell+1)\Delta t_h\right]} (t)~,
\end{align}
where $[ t \mapsto \rho^{n,h}_{\ell+1}(t) ]$ is the unique solution to~\eqref{eq:constrianedapprk}.

Roughly speaking, the present procedure consists in the application of two operators, $\Theta$ and $S$, at each fractional step $\left]\ell\Delta t_h, (\ell+1)\Delta t_h\right]$, $\ell \in \naturals$. The first operator gives ${\Xi^{n,h}_{\ell}} = \Theta[ \rho_\ell^{n,h}(\Delta t_h) ]$, while the second operator gives the solution $\rho_{\ell+1}^{n,h} = S [ \rho_\ell^{n,h}(\Delta t_h), {\Xi^{n,h}_{\ell}} ]$ of the constrained Cauchy problem of the form~\eqref{eq:constrianedapprk}, with $[x \mapsto \rho_\ell^{n,h}(\Delta t_h,x)]$ as initial datum and with $p^h({\Xi^{n,h}_{\ell}})$ as constraint.

More rigorously, for any $\rho^{n}_0 \in \mathcal{D}^n$ and $t \in \reals_+$, define recursively
\begin{align*}
    F^{n,h}[\rho^{n}_0](t) = S \left[\rho^{n}_0, \Theta\left[\rho^{n}_0\right]\right](t)
\end{align*}
if $t \in \left[0,\Delta t_h\right]$, and, if $t \in \left](\ell+1)\Delta t_h,(\ell+2)\Delta t_h\right]$, $\ell \in \naturals$, then
\begin{align*}
      F^{n,h}[\rho^{n}_0](t) =
      S\left[ F^{n,h}[\rho^{n}_0]\left((\ell+1) \Delta t_h\right),
      \Theta\left[F^{n,h}[\rho^{n}_0]\left((\ell+1) \Delta t_h\right) \right] \right](t) ~.
\end{align*}

\subsection{\textit{A priori} estimates}

In this section we prove that $\rho^{n,h}(t) = F^{n,h}[\rho^{n}_0](t)$ is in $\mathcal{D}^n$ on any bounded time interval $[0,T]$, $T>0$, and we estimate $\tv\left( \Psi\left( \rho^{n,h}(t) \right) \right)$ uniformly in $n$, $h$ and $t$.
To this aim, we introduce the following Temple functional
\begin{align}\label{eq:functionals}
    \Upsilon^{n,h}_T\left(t\right) &=
    \tv\left(\Psi\left(\rho^{n,h}(t)\right)\right)
    + \gamma^h \left(\rho^{n,h}(t),\Xi^{n,h}(t)\right)
    + \Gamma^h_T \left(t\right) ~,
\end{align}
with
\begin{align*}
    \gamma^h \left(\rho,\Xi\right) &= \left\{
    \begin{array}{l@{\qquad}l}
      0 &
        \begin{array}{l}
        \hbox{if }\rho(0-) > \bar\rho > \rho(0+) \hbox{ and}\\
        f^n\left(\rho(0\pm)\right) = p^h\left( \Xi(t) \right)
        \end{array}\\[10pt]
      4 \left[ f(\bar\rho) - p^h\left( \Xi \right)\right]
      &
        \begin{array}{l}
            \hbox{otherwise,}
        \end{array}
    \end{array}
    \right.\\
    \Gamma^h_T \left(t\right) &=
    5 \cdot2^{-h} ~f(\bar\rho) \left[ \frac{T}{\Delta t_h} - \left\lfloor \frac{t}{\Delta t_h} \right\rfloor \right],
\end{align*}
where $\lfloor\cdot\rfloor ~\colon~ \reals \to \integers$ denotes the floor function. Recall that the Temple functional adopted in Ref.~\cite{ColomboGoatinConstraint} involves the total variation of the approximating constraint, which is given \textit{a priori}. In our construction, at each fractional time interval we are dealing with a different approximating problem~\eqref{eq:constrianedapprk} and we need to know the solution at the previous step in order to fix the value of the constraint in~\eqref{eq:constrianedapprk2}. Therefore, the constraint $p\left(\Xi^{n,h}(t)\right)$, $t\in\reals_+$, and its total variation are not given \textit{a priori}. Nevertheless, due to the choice of $\Delta t_h$, we are able to bound the possible jump of $p\left(\Xi^{n,h}(t)\right)$ at each time step, as we will see in Lemma~\ref{lem:capraEcavoli}, and estimate \textit{a priori} the total variation of the efficiency. From this point of view, the functional $\Upsilon^{n,h}_T$ is the natural generalization of that one used in Ref.~\cite{ColomboGoatinConstraint}. In fact, the two functionals have in common the first two terms, namely
\begin{align}\label{eq:fanq}
    \mathcal{Q}^{n,h}\left(t\right) = \tv\left(\Psi\left(\rho^{n,h}(t)\right)\right) + \gamma^h \left(\rho^{n,h}(t),\Xi^{n,h}(t)\right),
\end{align}
while $\Gamma^h_T \left(t\right)$ is introduced to control the total variation of $p\left(\Xi^{n,h}(\cdot)\right)$ in the time interval $[t,T]$.

\begin{lemma}\label{lem:capraEcavoli}
    For any $\ell \in \naturals$, the jump in the efficiency at time $t = (\ell+1) \Delta t_h$, namely $\modulo{p^h(\Xi^{n,h}_{\ell+1}) - p^h(\Xi^{n,h}_{\ell})}$, is either zero or $2^{-h} f(\bar\rho)$.
\end{lemma}
\begin{proof}
    Fix $\ell \in \naturals$. If $\displaystyle \modulo{\Xi^{n,h}_{\ell+1} - \Xi^{n,h}_\ell} < \inf_{i=0,\ldots,m_h}\modulo{\xi_{i+1}^h - \xi_{i}^h}$, then $[\xi \mapsto p^h(\xi)]$ has at most one jump for $\xi$ between $\Xi^{n,h}_{\ell+1}$ and $\Xi^{n,h}_{\ell}$ and~\eqref{eq:CFL0} allows us to conclude. Because of~\eqref{eq:CFL1}, we just need to show
    \begin{align*}
        \modulo{\Xi^{n,h}_{\ell+1} - \Xi^{n,h}_\ell} \lip(p) < 2^{-h} f(\bar\rho) ~.
    \end{align*}
    By Proposition~\ref{prop:ws}, $\rho^{n,h}_{\ell+1}$ is a weak solution of the problem~\eqref{eq:constrianedapprk1}, \eqref{eq:constrianedapprk3} with $\rho^{n,h}_{\ell}(\Delta t_h)$ as initial condition. Then, for any $\phi$ in $\Cc1(\reals^2;\reals)$ we have
    \begin{align}\label{eq:rhoelleweaksol}
        &\int_{\reals_+}\int_{\reals} \left[ \rho^{n,h}_{\ell+1} ~\partial_t\phi + f(\rho^{n,h}_{\ell+1}) ~\partial_x\phi \right] {\d}x ~{\d}t
        +\int_{\reals} \rho^{n,h}_{\ell}(\Delta t_h,x) ~\phi(0,x) ~{\d}x =0 ~.
    \end{align}
    Let $(\eta_\nu)_\nu$ be a standard family of mollifiers and define $w_\nu=w*\eta_\nu$. Let $\delta_\varepsilon$ be as in~\eqref{eq:delta}. Take $0\le t_1 < t_2 \le \Delta t_h$ and consider the test function
    \begin{align*}
        \phi(t,x) &=
        \left[\int_{t-t_2+\varepsilon}^{t-t_1} \delta_\varepsilon(z) ~{\d}z\right]
        \left[\int_{x+\varepsilon}^{x+\iw} \delta_\varepsilon(z) ~{\d}z\right]
        w_\nu(x) ~.
    \end{align*}
    Observe that $\phi(0,\cdot) \equiv 0$ and that letting $\varepsilon$ go to zero we get
    \begin{align*}
        \partial_t\phi(t,x) &\to [\delta^D_{t_1}(t) - \delta^D_{t_2}(t)] ~\caratt{[-\iw,0]}(x) ~w_\nu(x) ~,\\
        \partial_x\phi(t,x) &\to \caratt{[t_1,t_2]}(t) ~[\delta^D_{-\iw+}(x) - \delta^D_{0-}(x)] ~w_\nu(x) ~.
    \end{align*}
    We pass to the limit in the Eq.~\eqref{eq:rhoelleweaksol} letting $\varepsilon$ go to zero and we obtain
    \begin{align*}
        &\int_{-\iw}^0 w_\nu(x) \left[ \rho^{n,h}_{\ell+1}(t_1,x) - \rho^{n,h}_{\ell+1}(t_2,x) \right] ~{\d}x\\
        =& \int_{t_1}^{t_2}\left[ w_\nu(0-)~ f\left(\rho^{n,h}_{\ell+1}(t,0-)\right)
        - w_\nu(-\iw+)~f\left(\rho^{n,h}_{\ell+1}(t,-\iw+)\right) \right] {\d}t ~.
    \end{align*}
    Then, as $\nu$ goes to infinity we get
    \begin{align}\label{eq:bellissima}
        \modulo{\int_{-\iw}^0 w(x) \left[\rho^{n,h}_{\ell+1}(t_1,x) - \rho^{n,h}_{\ell+1}(t_2,x)\right] ~{\d}x}
        \le (t_2-t_1) f(\bar\rho) w(0-) ~.
    \end{align}
    By~\eqref{eq:constrianedapprk3} and~\eqref{eq:constrianedapprk4}  we have that
    \begin{align*}
        \modulo{\Xi^{n,h}_{\ell+1} - \Xi^{n,h}_\ell}
        =& \modulo{\int_{\reals_-} w(x) \left[ \rho^{n,h}_{\ell+1}(\Delta t_h,x) - \rho^{n,h}_{\ell}(\Delta t_h,x)\right] {\d} x}\\
        =& \modulo{\int_{\reals_-} w(x) \left[ \rho^{n,h}_{\ell+1}(\Delta t_h,x) - \rho^{n,h}_{\ell+1}(0,x)\right] {\d} x}
        \le \Delta t_h ~f(\bar\rho) ~w(0-)~.
    \end{align*}
    Therefore by~\eqref{eq:CFL2} the proof is complete.
\end{proof}

We are ready to show that $\Upsilon^{n,h}_T$ is a Temple functional.

\begin{proposition}\label{prop:interactionbound}
    Let $h, n \in \naturals$ and $\rho^{n}_0 \in \mathcal{D}^n$. On $[0,T]$, the map $[t \mapsto \Upsilon^{n,h}_T\left(t\right) ]$ is non--increasing and it decreases by at least $2^{-n}f(\bar\rho)$ each time the number of waves increases.
\end{proposition}
\noindent The proof is deferred to Sec.~\ref{sec:6.1}.

In the next corollary we rely on Proposition~\ref{prop:interactionbound} to prove a uniform estimate on $\tv\left( \Psi\left( \rho^{n,h}(t) \right) \right)$.

\begin{corollary}\label{cor:estups2}
    There exists a constant $C>0$, that does not depend on $n$ or $h$, such that for all $t >0$
    \begin{align}\label{eq:etimateBV}
        \tv\left( \Psi\left( \rho^{n,h}(t) \right) \right)
        &\le \tv\left( \Psi\left( \rho_0 \right) \right) + 4f(\bar\rho) + C ~t ~.
    \end{align}
\end{corollary}

\begin{proof}
    We consider the functional $\mathcal{Q}^{n,h} = \Upsilon^{n,h}_T - \Gamma^h_T$ introduced in~\eqref{eq:fanq}. Proceeding as in the proof of Proposition~\ref{prop:interactionbound}, we can show that $\mathcal{Q}^{n,h}$ may increase only at $t \in \Delta t_h \naturals$. However, since $\Upsilon^{n,h}_T$ is strictly decreasing at $t \in \Delta t_h \naturals$, we have that for all $\ell \in \naturals$,
    \begin{align*}
        \mathcal{Q}^{n,h}(\ell\Delta t_h+) - \mathcal{Q}^{n,h}(\ell\Delta t_h-) \le \modulo{\Gamma^h_T(\ell\Delta t_h+) - \Gamma^h_T(\ell\Delta t_h-)} = 5 \cdot 2^{-h} f(\bar\rho) ~.
    \end{align*}
    Therefore, by~\eqref{eq:CFL2}
    \begin{align*}
        \tv\left( \Psi\left( \rho^{n,h}(t) \right) \right) &\le  \mathcal{Q}^{n,h}(t)
        \le \mathcal{Q}^{n,h}(0) + 5 \cdot 2^{-h} f(\bar\rho) \left\lfloor\frac{t}{\Delta t_h}\right\rfloor\\
        &\le \tv\left( \Psi\left( \rho_0 \right) \right) + 4f(\bar\rho) + 10 ~w(0-) ~\lip(p) ~f(\bar\rho) ~t ~,
    \end{align*}
    and the estimate~\eqref{eq:etimateBV} holds with $C = 10 ~w(0-) ~\lip(p) ~f(\bar\rho)$.
\end{proof}

By the results proved in Ref.~\cite{ColomboGoatinConstraint}, the assumption~\textbf{H2} and the corollary above, we have that both $\rho^{n,h}(t)$ and $\Xi^{n,h}(t)$ are well defined for any $t \in \left[0, T\right]$ and that $\rho^{n,h}$ belongs to $\C0(\reals_+; \mathcal{D}^n)$. In particular, $[ t \mapsto \rho^{n,h}(t) ]$ is piecewise constant with discontinuities along finitely many polygonal lines with bounded speed of propagation, that do not intersect each other at any time $t\in\Delta t_h\naturals$. By the construction of $\rho^{n,h}$ and its continuity with respect to time it is not difficult to show

\begin{proposition}\label{prop:lallelarsson}
The map $[t\mapsto \rho^{n,h}(t)]$ given by~\eqref{eq:apprsol} is an entropy weak solution in the sense of Definition~\ref{def:entropysol} (with $f^n,p^h$ replacing $f,p$) to the problem
\begin{subequations}\label{eq:constrianedapprGLOBAL}
\begin{align}\label{eq:constrianedapprGLOBAL1}
    \partial_t\rho + \partial_xf^n(\rho) &= 0 & (t,x) &\in \reals_+\times\reals\\
    \label{eq:constrianedapprGLOBAL2}
    f^n\left(\rho(t,0\pm)\right) &\le p^h\left( {\Xi^{n,h}}(t) \right) & t &\in \reals_+\\
    \label{eq:constrianedapprGLOBAL3}
    \rho(0,x) &=\rho_0^{n}(x) & x &\in \reals ~,
\end{align}
\end{subequations}
where $[t\mapsto \Xi^{n,h}(t)]$ is given by~\eqref{eq:appreff}.
\end{proposition}
\noindent The proof is deferred to Sec.~\ref{sec:lallelarsson}.

\begin{proposition}\label{prop:convergence}
    There exists a subsequence of $\rho^{n,h}$ converging a.e. on $\reals_+\times\reals$ to a limit $\rho \in \L\infty\left(\reals_+\times\reals; [0,R] \right)$. In addition, $\rho$ satisfies estimates~\eqref{eq:BVdepen} and~\eqref{eq:BVdepen2}.
\end{proposition}
\begin{proof}
    By the standard diagonal procedure argument, it is enough to prove convergence on an arbitrary time interval $[0,T]$, $T>0$. The sequence $\Psi\left(\rho^{n,h} \right)$ is uniformly bounded in $\L\infty\left([0,T]\times\reals;\reals\right) \cap \L\infty\left([0,T]; \BV(\reals;\reals)\right)$ by Corollary~\ref{cor:estups2}. In order to get compactness in $\Lloc1$, see Theorem~2.4 in Ref.~\cite{BressanBook}, we still need to show that $[t\mapsto\Psi\left(\rho^{n,h}(t,\cdot) \right)]$ is Lipschitz with respect to the $\L1$--norm. In analogy to~\eqref{eq:psi}, we define
    \begin{equation}\label{eq:psin}
        \Psi^n(\rho) = \sign(\rho-\bar\rho) ~ \left(f^n(\bar\rho)-f^n(\rho)\right)=\int_{\bar \rho}^\rho \modulo{\dot f^n(r)} ~{\d}r~.
    \end{equation}
    We observe that $\Psi$ coincides with $\Psi^n$ on $\mathcal{M}^n$ and, as a consequence, $\lip(\Psi^n) \le \lip(\Psi)$. Since $\rho^{n,h}$ takes values in $\mathcal{M}^n$, we have $\Psi^n\left(\rho^{n,h}\right) = \Psi\left(\rho^{n,h}\right)$. Hence, $\tv\left( \Psi^n\left( \rho^{n,h}(t) \right) \right) = \tv\left( \Psi\left( \rho^{n,h}(t) \right) \right)$ and by Corollary~\ref{cor:estups2} we have
    \begin{align*}
        \norma{\partial_x\Psi^n\left(\rho^{n,h}\right)}_{\L\infty\left([0,T]; \mathcal{M}_b(\reals;\reals)\right)}
        &\le C_T = \tv\left( \Psi\left( \rho_0 \right) \right) + 4f(\bar\rho) + C ~T~,
    \end{align*}
    uniformly in $n$ and $h$. Above, $\mathcal{M}_b(\reals;\reals)$ denotes the space of bounded Radon measures. Let $g^n = f^n \circ (\Psi^n)^{-1}$ and remark that by~\eqref{eq:psin}
    \begin{align*}
        \dot g^n(\psi) = \frac{\dot f^n \circ (\Psi^n)^{-1}(\psi)}{\dot \Psi^n \circ (\Psi^n)^{-1}(\psi)} = \frac{\dot f^n \circ (\Psi^n)^{-1}(\psi)}{\modulo{\dot f^n \circ (\Psi^n)^{-1}(\psi)}} \in \{-1,1\} ~.
    \end{align*}
    Hence $\partial_t\rho^{n,h}$ is bounded in $\L\infty\left([0,T];\mathcal{M}_b(\reals;\reals)\right)$ because, by Eq.~\eqref{eq:constrianedapprGLOBAL} and Theorem~4 in Ref.~\cite{MR624930}, we have
    \begin{align*}
        \norma{\partial_t\rho^{n,h}}_{\L\infty\left([0,T];\mathcal{M}_b(\reals;\reals)\right)} &\le \norma{\dot g^n}_{\L\infty([-f(\bar\rho),f(\bar\rho)];\reals)} \norma{\partial_x\Psi^n\left(\rho^{n,h}\right)}_{\L\infty\left([0,T];\mathcal{M}_b(\reals;\reals)\right)}\\
        &\le C_T~.
    \end{align*}
    As the functions $\Psi^n$ are uniformly Lipschitz, also the distributions $\mu^{n,h} =\partial_t \Psi^n(\rho^{n,h})$ are uniformly bounded measures in $\L\infty\left([0,T]; \mathcal{M}_b(\reals;\reals)\right)$ with
    \begin{align*}
        \norma{\mu^{n,h}}_{\L\infty\left([0,T]; \mathcal{M}_b(\reals;\reals)\right)} &\le
        \lip(\Psi^n) ~C_T~.
    \end{align*}
    Now, let $(\eta_\nu)_\nu$ be a standard family of mollifiers in $\Cc\infty(\reals^2; \reals )$ and define $F^{n,h}_\nu = \Psi^n\left(\rho^{n,h}\right)*\eta_\nu$ and $\mu^{n,h}_\nu = \mu^{n,h}*\eta_\nu$. Then
    \begin{equation*}
        \norma{\mu^{n,h}_\nu}_{\L\infty([0,T]; \L1(\reals;\reals))} \le
        \norma{\mu^{n,h}}_{\L\infty([0,T]; \mathcal{M}_b(\reals;\reals))} ~.
    \end{equation*}
    Due to the regularity of $F^{n,h}_\nu$, for any $\delta >0$ and for any $0\le t<t+ \delta\le T$
    \begin{align*}
        \norma{F^{n,h}_\nu (t+\delta,\cdot) - F^{n,h}_\nu (t,\cdot)}_{\L1(\reals;\reals)}
        &= \int_{\reals} \modulo{ \int_t^{t+\delta} \mu^{n,h}_\nu(s,x) ~{\d}s} ~{\d}x\\
        &\le
        \delta \norma{\mu^{n,h}_\nu}_{\L\infty([0,T]; \L1(\reals;\reals))} \le
        \delta ~\lip(\Psi^n) ~C_T~,
    \end{align*}
    and as $\nu$ go to zero we deduce the uniform Lipschitz continuity in time of $\Psi\left(\rho^{n,h} \right) = \Psi^n\left(\rho^{n,h} \right)$:
    \begin{equation*}
        \norma{\Psi(\rho^{n,h}(t+\delta,\cdot)) - \Psi(\rho^{n,h}(t,\cdot))}_{\L1(\reals;\reals)} \le \delta ~\lip(\Psi) ~C_T ~.
    \end{equation*}
    In this way we prove the existence of a subsequence of $\Psi\left(\rho^{n,h} \right) = \Psi^n\left(\rho^{n,h} \right)$ that converges in $\Lloc1([0,T]\times\reals;\reals)$ to a function $\psi$ in $\L\infty\left([0,T]; \BV(\reals; [-f(\bar\rho), f(\bar\rho)]) \right)$ which satisfies
    \begin{align}\label{eq:boris}
        \norma{\psi(t+\delta,\cdot) - \psi(t,\cdot)}_{\L1(\reals;\reals)} \le \delta ~\lip(\Psi) ~C_T ~.
    \end{align}
    For simplicity we still denote the subsequence $\Psi\left(\rho^{n,h} \right)$. Since $\Psi$ is invertible and $\Psi^{-1}$ is continuous, also $\rho^{n,h}$ converges in $\Lloc1([0,T]\times\reals;\reals)$ to a function $\rho = \Psi^{-1}(\psi)$ in $\L\infty\left([0,T]\times\reals; [0,R] \right)$. In particular, by~\eqref{eq:etimateBV} and~\eqref{eq:boris} the estimates~\eqref{eq:BVdepen} and~\eqref{eq:BVdepen2} hold true.
\end{proof}

\begin{lemma}\label{lem:besancon}
    For any $T>0$
    \begin{align*}
        \lim_{n,h \to + \infty} \int_0^T \modulo{\Xi^{n,h}(t) - \int_{\reals_-} w(x) ~\rho(t,x) ~{\d} x } ~{\d}t =0~.
    \end{align*}
\end{lemma}

\begin{proof}
    Let $T>0$ and define $\ell_T^h = \lfloor T/\Delta t_h\rfloor$. Then by~\eqref{eq:appreff} and~\eqref{eq:bellissima}
    \begin{align*}
        &\int_0^T \modulo{ \Xi^{n,h}(t) - \int_{\reals_-} w(x) ~\rho^{n,h}(t,x) ~{\d}x } {\d}t \le\\
        \le& \sum_{\ell=0}^{\ell_T^h-1} \int_{\ell\Delta t_h}^{(\ell+1)\Delta t_h} \modulo{\Xi^{n,h}_\ell(t - \ell\Delta t_h) - \int_{\reals_-} w(x) ~\rho^{n,h}_{\ell+1}(t - \ell\Delta t_h,x) ~{\d}x} ~{\d}t\\
        &+ \int_{\ell_T^h\Delta t_h}^T \modulo{\Xi^{n,h}_{\ell_T^h}(t - \ell_T^h\Delta t_h) - \int_{\reals_-} w(x) ~\rho^{n,h}_{\ell_T^h+1}(t - \ell_T^h\Delta t_h,x) ~{\d}x} ~{\d}t\\
        =& \sum_{\ell=0}^{\ell_T^h-1} \int_{0}^{\Delta t_h} \modulo{\int_{\reals_-} w(x) \left[\rho^{n,h}_{\ell+1}(0,x) - \rho^{n,h}_{\ell+1}(t,x)\right] {\d}x} ~{\d}t\\
        &+ \int_{0}^{T-\ell_T^h\Delta t_h} \modulo{\int_{\reals_-} w(x) \left[\rho^{n,h}_{\ell_T^h+1}(0,x) - \rho^{n,h}_{\ell_T^h+1}(t,x)\right] ~{\d}x} ~{\d}t\\
        \le& \left[\sum_{\ell=0}^{\ell_T^h-1} \int_{0}^{\Delta t_h} t ~{\d}t
        + \int_{0}^{T-\ell_T^h\Delta t_h} t ~{\d}t\right] f(\bar\rho) ~w(0-)\\
        =&\frac{\Delta t_h^2 ~\ell_T^h + (T-\ell_T^h\Delta t_h)^2}{2} ~f(\bar\rho) ~w(0-) ~.
    \end{align*}
    Therefore, since $\ell_T^h\Delta t_h$ converges to $T$ as $h$ goes to infinity and $\rho^{n,h}$ converges to $\rho$ in $\Lloc1(\reals_+\times\reals;\reals)$ as $n$ and $h$ go to infinity, the proof is complete.
\end{proof}

Since  $\rho^{n,h}$ converges to $\rho$  in $\Lloc1$, Proposition~\ref{prop:lallelarsson} and Lemma~\ref{lem:besancon} imply that $\left[t \mapsto \rho(t)\right]$ satisfies the conditions~\eqref{eq:entropysol1}--\eqref{eq:entropysol3} and~\eqref{eq:entropysol5} of Definition~\ref{def:entropysol} with respect to the problem~\eqref{eq:constrianed}. Moreover, $\rho$ satisfies the condition~\eqref{eq:entropysol4} of Definition~\ref{def:entropysol} by Lemma~\ref{lem:star}, and it satisfies estimate~\eqref{eq:BVdepen} and~\eqref{eq:BVdepen2} by~ Proposition~\ref{prop:convergence}. Finally, observe that $\rho \in \C0\left(\reals_+; \Lloc1(\reals;[0,R])\right)$  because of entropy inequalities~\eqref{eq:entropysol1}--\eqref{eq:entropysol3}, see Ref.~\cite{MR2780572}. As already observed $\Psi\left(\rho\right) \in \C0\left(\reals_+; \BV(\reals; \reals) \right)$ thus $\rho(t) \in \mathcal{D}$ for all $t$.

Uniqueness of the entropy weak solutions to the Cauchy problem~\eqref{eq:constrianed} in the case $p\in\Lip([0,R]; \reals)$, $\rho_0 \in \mathcal{D}$ follows directly from uniqueness in the $\L\infty$--framework, see the first part of the proof of Theorem~\ref{thm:1}.

\section{The constrained Riemann problem}\label{sec:Riemann}

In this section we study constrained Riemann problems of the form
\begin{subequations}\label{eq:constrianedRiemann}
\begin{align}\label{eq:constrianedRiemann1}
    \partial_t\rho + \partial_xf(\rho) &= 0 & (t,x) &\in \reals_+\times\reals\\
    \label{eq:constrianedRiemann2}
    f\left(\rho(t,0\pm)\right) &\le p\left( \int_{\reals_-} w(x) ~ \rho(t,x) ~{\d} x \right) & t &\in \reals_+\\
    \label{eq:constrianedRiemann3}
    \rho(0,x) &=\left\{
    \begin{array}{l@{\qquad\hbox{if }}l}
      \rho_L & x < 0\\
      \rho_R & x \ge 0
    \end{array}
    \right. & x &\in \reals
\end{align}
\end{subequations}
with $\rho_L, \rho_R \in [0,R]$. Along with~\textbf{(F)} and~\textbf{(W)}, we assume that:
\begin{enumerate}
  \item[\textbf{(P2)}] $p$ belongs to $\PC \left( \left[0,R\right]; \left]0,f(\bar\rho)\right] \right)$ and is a non--increasing map.
\end{enumerate}
The assumption~\textbf{(P2)} is introduced in place of~\textbf{(P1)} to allow an explicit construction of solutions to~\eqref{eq:constrianedRiemann}. However, the regularity of $p$ required by~\textbf{(P2)} is not enough to apply the results of Theorem~\ref{thm:1}. In fact, the uniqueness of entropy weak solutions as well as the stability estimate~\eqref{eq:lipdepen-localized} do not hold in the present framework, as we will see in Example~\ref{ex:nonuniqueness1}.

Aiming for a general construction of the solutions to~\eqref{eq:constrianedRiemann}, we allow $p$ to be a multi--valued piecewise constant function, namely, see Fig.~\ref{fig:Korn}, right:
\begin{itemize}
\item there exist $\xi_1, \ldots, \xi_n \in \left]0,R\right[$ and $p_0, \ldots, p_n \in \left]0,f(\bar\rho)\right]$, with $\xi_i < \xi_{i+1}$ and $p_i > p_{i+1}$, such that $p(0) = p_0$, $p(R) = p_n$, $p ~\caratt{\left]\xi_i, \xi_{i+1}\right[} = p_i$ for $i=0,\ldots,n$,  $p(\xi_i)= \left[p_i, p_{i-1}\right]$ for $i=1,\ldots,n$, being $\xi_0 = 0$ and $\xi_{n+1} = R$.
\end{itemize}
Let $\sigma(\rho_L, \rho_R)=\left( f(\rho_L) - f(\rho_R)\right)/ \left(\rho_L - \rho_R\right)$ be the speed of propagation of a shock between $\rho_L$ and $\rho_R$, while $\lambda(\rho) = f'(\rho)$ is the characteristic speed. Introduce the maps $\check\rho,\hat\rho ~\colon~ [0,f(\bar\rho)] \to [0,R]$ implicitly defined by
\begin{align*}
    f\left(\check\rho(p)\right) = p = f\left(\hat\rho(p)\right)\quad\hbox{ and }\quad\check\rho(p) \le\bar\rho \le \hat\rho(p)~.
\end{align*}
Let $\check\rho_i = \check\rho(p_i)$ and $\hat\rho_i = \hat\rho(p_i)$. Denote by $\mathcal{R}$ the classical Riemann solver Refs.~\cite{Lax1957}, \cite{Liu1975}. This means that the map $\left[ (t,x) \mapsto \mathcal{R}[\rho_L, \rho_R](x/t) \right]$ is the unique entropy weak solution for the unconstrained problem~\eqref{eq:constrianedRiemann1}, \eqref{eq:constrianedRiemann3}, see Refs.~\cite{BressanBook}, \cite{Rosinibook}, \cite{SerreBook} for its construction. As we will see in Proposition~\ref{prop:riemann}, the classical solutions given by $\mathcal{R}$ may not satisfy the constraint~\eqref{eq:constrianedRiemann2}. For this reason we consider also nonclassical solutions, namely solutions that do not satisfy the Lax entropy inequalities, see Ref.~\cite{LeflochBook} as a general reference. In general, entropy weak solutions to~\eqref{eq:constrianedRiemann} are not self--similar nor unique, as we will show in the two following examples.
\begin{example}\label{ex:es0}
    Let $0<\rho_L<\rho_R<R$ be such that $f(\rho_L) > f(\rho_R)$. If $\xi_i \le \rho_L < \xi_{i+1}$, $p_{i+1} < f(\rho_R) < p_i$ and $j>i$ is such that  $f\left(\hat\rho_j\right) = p\left(\hat\rho_j\right)$ and $f\left(\hat\rho_k\right) > p\left(\hat\rho_k\right)$ for all $k \in \{i,\ldots,j-1\}$, see Fig.~\ref{fig:nonself}, left, then the entropy weak solution to the corresponding Riemann problem~\eqref{eq:constrianedRiemann} is not self--similar, see Fig.~\ref{fig:nonself}, right. More in detail, for sufficiently small times, the solution corresponds to the classical one and is given by a shock with speed $\sigma(\rho_L,\rho_R)<0$. The corresponding map $[t \mapsto \xi(t)]$ is increasing and by hypothesis, there exists a time $t_1 < -\iw/ \sigma(\rho_L,\rho_R)$ such that $\xi(t_1) = \xi_{i+1}$. Then, the efficiency of the exit falls to $p_{i+1}$ and the solution given by the classical Riemann solver $\mathcal{R}$ no longer satisfies the constraint condition~\eqref{eq:constrianedRiemann2}. As a result, at time $t_1$ the solution performs a nonclassical discontinuity at the constraint location and two further classical shocks appear, one with speed $\sigma(\rho_R, \hat\rho_{i+1})<0$ and one with speed $\sigma(\check\rho_{i+1}, \rho_R)>0$. The final solution can then be constructed by taking into account the interactions between the shocks on each side of the constraint and the appearance of new shocks each time $[t \mapsto \xi(t)]$ crosses $\xi_k$, $k \in \{i+1,\ldots,j-1\}$.
\end{example}
\begin{figure}[htpb]
      \centering
        \includegraphics[width=.75\textwidth]{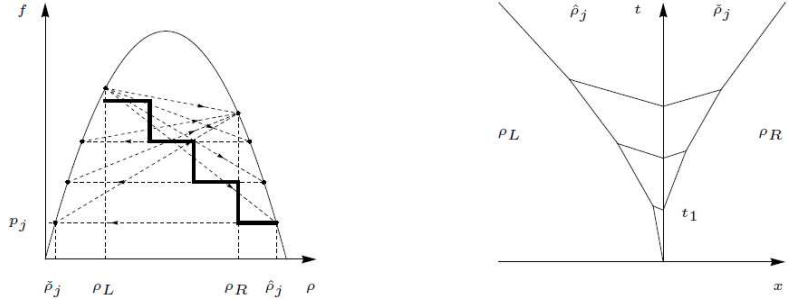}
      \caption{Construction of a non self--similar entropy weak solution as in Example~\ref{ex:es0}. On the left, the thick line corresponds to the efficiency of the exit $p|_{]\rho_L,\hat\rho_j[}$.}
\label{fig:nonself}
\end{figure}
As we have seen, the lack of self--similarity is related to the jumps of $[t \mapsto p\left(\xi(t)\right)]$. Nevertheless, in the proof of Proposition~\ref{prop:riemann} we show that any entropy weak solution of~\eqref{eq:constrianedRiemann} is self--similar for sufficiently small times. Therefore, it makes sense to introduce nonclassical \emph{local} Riemann solvers, see Definition~\ref{def:Rp}. Then, the availability of a local Riemann solver allows us to construct a \emph{global} solution to the Riemann problem~\eqref{eq:constrianedRiemann} by a wave--front tracking algorithm in which the jumps in the map $[t \mapsto p\left(\xi(t)\right)]$ are interpreted as interactions.

The next example shows that the entropy weak solutions to the constrained Riemann problem~\eqref{eq:constrianedRiemann} are not necessarily unique.
\begin{figure}[htpb]
      \centering
        \includegraphics[width=1\textwidth]{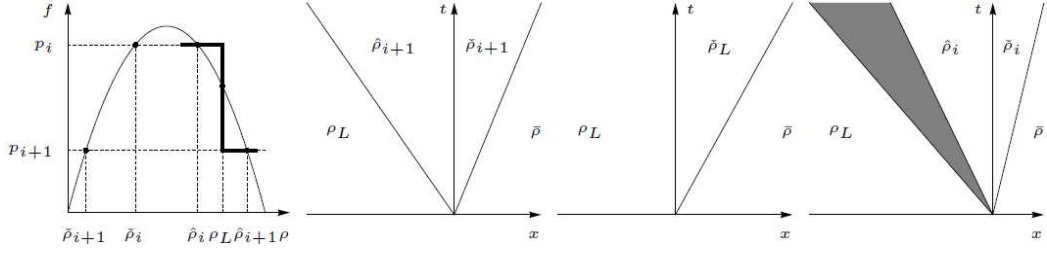}
      \caption{With reference to Example~\ref{ex:nonuniqueness1}, the flux configuration and three different solutions $\rho_1$, $\rho_2$ and $\rho_3$ to the same Riemann problem are represented from left to right. Here $\check\rho_L = \check\rho\left(f(\rho_L)\right)$.}
\label{fig:nonuiqueness1}
\end{figure}
\begin{example}\label{ex:nonuniqueness1}
    Consider the constrained Riemann problem~\eqref{eq:constrianedRiemann} with $\rho_L = \xi_{i+1} \in \left]\bar\rho, R\right[$ and $\rho_R = \bar\rho$. Assume that $f\left( \hat\rho_{i+1} \right) = p_{i+1} \le f(\xi_{i+1}) \le p_i = f\left( \hat\rho_i \right) < f(\bar\rho)$, see Fig.~\ref{fig:nonuiqueness1}, left, then
    \begin{align*}
        \rho_1(x/t) &= \left\{
        \begin{array}{l@{\quad\hbox{if }}l}
           \mathcal{R}[\xi_{i+1}, \hat\rho_{i+1}](x/t) & x <0\\
           \mathcal{R}[\check\rho_{i+1}, \bar\rho](x/t) & x \ge0~,
        \end{array}
        \right.
        \\
        \rho_2(x/t) &= \left\{
        \begin{array}{l@{\quad\hbox{if }}l}
           \xi_{i+1} & x <0\\
           \mathcal{R}[\check\rho\left(f(\xi_{i+1})\right), \bar\rho](x/t) & x \ge0~,
        \end{array}
        \right.
        \\
        \rho_3(x/t) &= \left\{
        \begin{array}{l@{\quad\hbox{if }}l}
           \mathcal{R}[\xi_{i+1}, \hat\rho_i](x/t) & x <0\\
           \mathcal{R}[\check\rho_i, \bar\rho](x/t) & x \ge0~,
        \end{array}
        \right.
    \end{align*}
    are self--similar entropy weak solutions of problem~\eqref{eq:constrianedRiemann} with the same datum, see Fig.~\ref{fig:nonuiqueness1}. Clearly, the above solutions are distinct if $p_{i+1} \ne f(\xi_{i+1}) \ne p_i$, otherwise two of them may coincide. Additionally, for an arbitrarily chosen $\bar t>0$, the functions
    \begin{align*}
        \rho_{\bar t,1}(t,x) &= \left\{
        \begin{array}{l@{\quad}l}
           \rho_2\left(x/t\right) &\hbox{if }0<t \le\bar t\\
           \rho_1\left(x/(t-\bar t)\right) &\hbox{if }t >\bar t\hbox{ and }x<0\\
           \check\rho_{i+1} &\hbox{if } \bar t < t \le \tilde t_1 \hbox{ and }0\le x <\sigma\left(\check\rho_{i+1},\check\rho\left(f(\xi_{i+1})\right)\right) (t-\bar t)\\
           \bar\rho &\hbox{if } \bar t < t \le \tilde t_1 \hbox{ and }x\ge \sigma\left(\check\rho\left(f(\xi_{i+1})\right), \bar\rho\right) t\\
           \check\rho_{i+1} &\hbox{if } t> \tilde t_1 \hbox{ and }0\le x <\sigma\left(\check\rho_{i+1},\bar\rho\right) (t-\tilde t_1) + \tilde x_1\\
           \bar\rho &\hbox{if } t> \tilde t_1 \hbox{ and } x \ge\sigma\left(\check\rho_{i+1},\bar\rho\right) (t-\tilde t_1) + \tilde x_1\\
           \check\rho\left(f(\xi_{i+1})\right) &\hbox{otherwise},
        \end{array}
        \right.
        \\
        \rho_{\bar t,3}(t,x) &= \left\{
        \begin{array}{l@{\quad}l}
           \rho_2\left(x/t\right) &\hbox{if }0<t \le\bar t\\
           \rho_3\left(x/(t-\bar t)\right) &\hbox{if }t >\bar t\hbox{ and }x<0\\
           \check\rho_i &\hbox{if } \bar t < t \le \tilde t_3 \hbox{ and }0\le x <\sigma\left(\check\rho_i,\check\rho\left(f(\xi_{i+1})\right)\right) (t-\bar t)\\
           \bar\rho &\hbox{if } \bar t < t \le \tilde t_3 \hbox{ and }x\ge \sigma\left(\check\rho\left(f(\xi_{i+1})\right), \bar\rho\right) t\\
           \check\rho_i &\hbox{if } t> \tilde t_3 \hbox{ and }0\le x <\sigma\left(\check\rho_i,\bar\rho\right) (t-\tilde t_3) + \tilde x_3\\
           \bar\rho &\hbox{if } t> \tilde t_3 \hbox{ and } x \ge\sigma\left(\check\rho_i,\bar\rho\right) (t-\tilde t_3) + \tilde x_3\\
           \check\rho\left(f(\xi_{i+1})\right) &\hbox{otherwise},
        \end{array}
        \right.
    \end{align*}
    where
    \begin{align*}
        \tilde t_1 &= \dfrac{\bar t~\sigma\left(\check\rho_{i+1}, \check\rho\left(f(\xi_{i+1})\right)\right)}{\sigma\left(\check\rho_{i+1}, \check\rho\left(f(\xi_{i+1})\right)\right) - \sigma\left(\check\rho\left(f(\xi_{i+1})\right),\bar\rho\right)}~,&
        \tilde x_1 &= \tilde t_1~\sigma\left(\check\rho\left(f(\xi_{i+1})\right),\bar\rho\right)~,
        \\
        \tilde t_3 &= \dfrac{\bar t~\sigma\left(\check\rho_i, \check\rho\left(f(\xi_{i+1})\right)\right)}{\sigma\left(\check\rho_i, \check\rho\left(f(\xi_{i+1})\right)\right) - \sigma\left(\check\rho\left(f(\xi_{i+1})\right),\bar\rho\right)}~,&
        \tilde x_3 &= \tilde t_3~\sigma\left(\check\rho\left(f(\xi_{i+1})\right),\bar\rho\right)~,
    \end{align*}
    are also entropy weak solutions, see Fig.~\ref{fig:nonuiqueness2}. Therefore, because of the arbitrariness of $\bar t$, we can build infinitely many different solutions which are not self--similar on any open time interval. However, remark that the asymptotic profile of $\rho_{\bar t,i}$ coincides with the asymptotic profile of $\rho_i$.

    One may guess that the lack of uniqueness is due to the fact that $p$ is a multi--valued function. But we observe that if we pick up $f(\xi_{i+1})$ as value for $p(\xi_{i+1})$, then still all the above solutions are admissible. Moreover, if we have $p(\xi_{i+1}) \neq f(\xi_{i+1})$, then both $\rho_1$ and $\rho_3$ are admissible, but not $\rho_2$. Thus there is more than one entropy weak solution even if $p$ is a single valued function.
\end{example}
\begin{figure}[htpb]
      \centering
        \includegraphics[width=.55\textwidth]{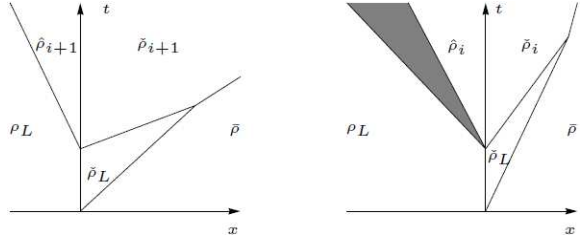}
      \caption{With reference to Example~\ref{ex:nonuniqueness1}, the solutions $\rho_{\bar t,1}$ and $\rho_{\bar t,3}$. Above $\check\rho_L = \check\rho\left(f(\rho_L)\right)$.}
\label{fig:nonuiqueness2}
\end{figure}

Introduce the subset of $\left[0, R\right]^2$
\begin{align*}
    \mathcal{C} &= \left\{ \left(\rho_L, \rho_R\right) \in \left[0, R\right]^2 ~\colon~ \left(\rho_L, \rho_R\right)\hbox{satisfies condition~\textbf{(C)}} \right\} ,
\end{align*}
where we say that $\left(\rho_L, \rho_R\right)$ satisfies condition~\textbf{(C)} if it satisfies one of the following conditions:
\begin{description}
  \item[(C1)] $\rho_L < \rho_R$, $f(\rho_R) < f(\rho_L)$ and $f(\rho_R) \le p(\rho_L+)$;
  \item[(C2)] $\rho_L <\rho_R$, $f(\rho_L) \le f(\rho_R)$ and $f(\rho_L) \le p(\rho_L+)$;
  \item[(C3)] $\rho_R \le \rho_L \le \bar\rho$ and $f(\rho_L) \le p\left( \rho_L +\right)$;
  \item[(C4)] $\rho_R \le \bar\rho < \rho_L$ and $f(\bar\rho) = p\left( \rho_L +\right)$;
  \item[(C5)] $\bar\rho < \rho_R \le \rho_L$, $f(\rho_R) \le p\left( \rho_L -\right)$ and $f(\rho_L) < p\left( \rho_L +\right)$.
\end{description}
In Proposition~\ref{prop:riemann} we will prove that a constrained Riemann problem admits as unique entropy weak solution the classical one, at least for small times, if and only if its initial datum satisfies condition~\textbf{(C)}.

Analogously, introduce the subset of $\left[0, R\right]^2$
\begin{align*}
    \mathcal{N}  &= \left\{ \left(\rho_L, \rho_R\right) \in \left[0, R\right]^2 ~\colon~ \left(\rho_L, \rho_R\right)\hbox{satisfies condition~\textbf{(N)}} \right\},
\end{align*}
where we say that $\left(\rho_L, \rho_R\right)$ satisfies condition~\textbf{(N)} if it satisfies one of the following conditions:
\begin{description}
    \item[(N1)] $\rho_L < \rho_R$ and $f(\rho_L)> f(\rho_R) > p(\rho_L+)$;
    \item[(N2)] $\rho_L <\rho_R$, $f(\rho_L) \le f(\rho_R)$ and $f(\rho_L)> p(\rho_L-)$;
    \item[(N3)] $\rho_R \le \rho_L \le \bar\rho$ and $f(\rho_L) > p\left( \rho_L -\right)$;
    \item[(N4a)] $\rho_R \le \bar\rho < \rho_L$, $f(\bar\rho) \ne p\left( \rho_L- \right)$ and $f(\rho_L) < p(\rho_L+)$;
    \item[(N4b)] $\rho_R \le \bar\rho < \rho_L$, $f(\bar\rho) \ne p\left( \rho_L- \right)$ and $f(\rho_L)> p(\rho_L-)$;
    \item[(N5a)] $\bar\rho < \rho_R \le \rho_L$, $f(\rho_R) > p\left( \rho_L- \right)$ and $f(\rho_L) < p\left( \rho_L +\right)$;
    \item[(N5b)] $\bar\rho < \rho_R \le \rho_L$ and $f(\rho_L) > p\left( \rho_L -\right)$.
\end{description}
In Proposition~\ref{prop:riemann} we will prove that, at least for small times, a constrained Riemann problem has a unique entropy weak solution which is nonclassical if and only if its initial datum satisfies condition~\textbf{(N)}.

Observe that if the constraint function $p$ is constant in a neighborhood of the state $\rho_L$, then $p(\rho_L-) = p(\rho_L+)$ and this simplifies the above conditions. Also a right or left continuity assumption on $p$ would simplify the above conditions. However, we keep $p$ as a multi--valued function to take into account all the possible solutions of~\eqref{eq:constrianedRiemann}.

In the next proposition, we show that uniqueness holds if and only if the initial data are in $\mathcal{C} \cup \mathcal{N}$.
\begin{proposition}\label{prop:riemann}
    Consider the constrained Riemann problem~\eqref{eq:constrianedRiemann}.

    \noindent$\bullet$~If $(\rho_L, \rho_R) \in \mathcal{C}$, then the map $\left[ (t,x) \mapsto \mathcal{R}[\rho_L, \rho_R](x/t) \right]$ is the unique entropy weak solution at least  for $t>0$ sufficiently small.

    \noindent$\bullet$~If $(\rho_L, \rho_R) \in \mathcal{N}$, then there exists a unique $\bar p \in \left[p(\rho_L+), p(\rho_L-)\right]$ such that the map
            \begin{align*}
                \left[
                t \mapsto \left\{\begin{array}{l@{\quad\hbox{if }}l} \mathcal{R}[\rho_L, \hat\rho\left( \bar p \right)](x/t)&x<0\\ \mathcal{R}[\check\rho\left( \bar p \right), \rho_R](x/t)&x\ge0\end{array}\right.\right]
            \end{align*}
            is the unique entropy weak solution at least for $t>0$ sufficiently small.

    \noindent$\bullet$~If $(\rho_L, \rho_R) \in [0,R]^2 \setminus (\mathcal{C} \cup \mathcal{N})$, then the corresponding constrained Riemann solver~\eqref{eq:constrianedRiemann} admits more than one entropy weak solution.
\end{proposition}
\begin{proof}
We stress that any nonclassical entropy weak solution in the sense of Definition~\ref{def:entropysol} is also a classical entropy weak solution in the Kru\v zkov sense in the half--planes $\reals_+ \times \reals_-$ and $\reals_+ \times \reals_+$. Therefore, at least for $t>0$ sufficiently small, by Proposition~\ref{prop:ws} and assumption~\textbf{(P2)} any nonclassical entropy weak solution of~\eqref{eq:constrianedRiemann} must have the form, see Fig.~\ref{fig:N123},
\begin{subequations}\label{eq:nonclassicalsol}
\begin{align}\label{eq:nonclassicalsol1}
    \rho(t,x) = \left\{
    \begin{array}{l@{\qquad\hbox{if }}l}
      \mathcal{R}[\rho_L,\hat\rho(\bar p)](x/t)& x<0\\
      \mathcal{R}[\check\rho(\bar p),\rho_R](x/t)& x\ge0 ~.
    \end{array}
    \right.
\end{align}

\begin{figure}[htpb]
      \centering
        \includegraphics[width=.9\textwidth]{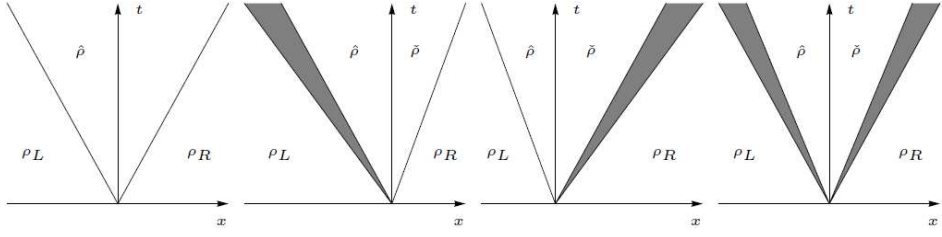}
      \caption{The four possible configurations of nonclassical entropy weak solutions of the form~\eqref{eq:nonclassicalsol}.}
\label{fig:N123}
\end{figure}
\noindent Observe that~\eqref{eq:nonclassicalsol1} is uniquely identified once we know $\bar p$ which, by \eqref{eq:ws}, satisfies
   \begin{align}\label{eq:nonclassicalsol2}
        \bar p = f\left(\check\rho(\bar p)\right) = f\left(\hat\rho(\bar p)\right) ~.
    \end{align}
We recall that~\eqref{eq:nonclassicalsol2} means in particular that the Rankine--Hugoniot jump condition is satisfied at $x=0$ even when the solution to the Riemann problem is nonclassical. As a consequence of~\eqref{eq:nonclassicalsol2}, of assumption~\textbf{(P2)} and of the continuity of $\left[t \mapsto \xi(t)\right]$, we have that
\begin{align}\label{eq:nonclassicalsol3}
    \bar p \in \left[p(\rho_L+), p(\rho_L-)\right].
\end{align}
\end{subequations}
This implies that if $p(\rho_L+) = p(\rho_L-)$, then $p(\xi)$ is constant in a neighborhood of $\rho_L$ and, since the solution is in $\C0\left(\reals_+; \Lloc1(\reals;[0,R])\right)$, uniqueness is ensured by the results in Ref.~\cite{ColomboGoatinConstraint}. However, the continuity of $p$ at $\rho_L$ is not a necessary condition for uniqueness. In Sec.~\ref{sec:proofRiem} we prove that:
\begin{description}
  \item[$(\rho_L, \rho_R) \in \mathcal{C}$] In this case, the corresponding classical solution satisfies~\eqref{eq:constrianedRiemann} for all $t>0$ sufficiently small and it is not possible to construct a different solution.
  \item[$(\rho_L, \rho_R) \in \mathcal{N}$] In this case, the corresponding classical solution does not satisfy~\eqref{eq:constrianedRiemann2}, and there exists a unique nonclassical solution that satisfies~\eqref{eq:constrianedRiemann}.
\end{description}

Now we list the ``pathological'' cases, where we have more than one admissible solution. We stress once again that a necessary condition for non--uniqueness is  $p(\rho_L-) \ne p(\rho_L+)$ and $p(\rho_L-) \ge f(\rho_L) \ge p(\rho_L+)$. It is important to stress that in general the solutions to the constrained Riemann problem~\eqref{eq:constrianedRiemann} are not self--similar, see Example~\ref{ex:es0}. All the cases listed below describe self--similar solutions because we let the solutions evolve only on a small interval of time.
\begin{description}
\item[(CN2)] If $\rho_L < \rho_R$, $f(\rho_L) \le f(\rho_R)$ and $p(\rho_L+)< f(\rho_L) \le p(\rho_L-)$, then the classical solution $\left[ (t,x) \mapsto \mathcal{R}[\rho_L, \rho_R](x/t) \right]$, which consists of a shock with non negative speed, as well as the nonclassical solution~\eqref{eq:nonclassicalsol}, with $\bar p = p(\rho_L+)$, are distinct solutions of~\eqref{eq:constrianedRiemann}.
\item[(CN3)] If $\rho_R \le \rho_L\le \bar\rho$ and  $p(\rho_L+)< f(\rho_L) \le p(\rho_L-)$, then the classical solution $\left[ (t,x) \mapsto \mathcal{R}[\rho_L, \rho_R](x/t) \right]$, which consists of a possible null rarefaction on the right of the constraint, as well as the nonclassical solution~\eqref{eq:nonclassicalsol}, with $\bar p = p(\rho_L+)$, are distinct solutions of~\eqref{eq:constrianedRiemann}.
\item[(NNN4)] If $\rho_R \le \bar\rho < \rho_L$, $p(\rho_L-) \ne p(\rho_L+)$ and $p(\rho_L+) \le f(\rho_L) \le p(\rho_L-)$, then the nonclassical solutions of the form~\eqref{eq:nonclassicalsol} corresponding to $\bar p \in \{p(\rho_L+), f(\rho_L),  p(\rho_L-)\}$ satisfy~\eqref{eq:constrianedRiemann}, see Example~\ref{ex:nonuniqueness1}. Observe that such solutions are distinct as far as they correspond to distinct constraint levels $\bar p$, and that in any case there exist at least two distinct nonclassical solutions.
\item[(CNN5)] If $\bar\rho < \rho_R \le \rho_L$, $f(\rho_R) \le p(\rho_L-)$, $p(\rho_L-) \ne p(\rho_L+)$ and $p(\rho_L+) \le f(\rho_L)$, then the classical solution $\left[ (t,x) \mapsto \mathcal{R}[\rho_L, \rho_R](x/t) \right]$, which consists of a possible null rarefaction on the left of the constraint, as well as the nonclassical solutions of the form~\eqref{eq:nonclassicalsol} corresponding to $\bar p \in \{p(\rho_L+), f(\rho_L)\}$ satisfy~\eqref{eq:constrianedRiemann}. Observe that the two nonclassical solutions are distinct as far as they correspond to distinct constraint levels $\bar p$, and that in any case there exist at least two distinct solutions, one classical and one nonclassical.
\item[(NNN5)] If $\bar\rho < \rho_R < \rho_L$, $f(\rho_R) > p(\rho_L-) \ge f(\rho_L) \ge p(\rho_L+)$ and $p(\rho_L-) \ne p(\rho_L+)$, then the nonclassical solutions of the form~\eqref{eq:nonclassicalsol} corresponding to $\bar p \in \{p(\rho_L+), f(\rho_L),  p(\rho_L-)\}$ satisfy~\eqref{eq:constrianedRiemann}. Observe that such solutions are distinct as far as they correspond to distinct constraint levels $\bar p$, and that in any case there exist at least two distinct nonclassical solutions.
\end{description}
This concludes the proof.
\end{proof}

As the local solutions of the Riemann problem are not unique in general, we are naturally led to question the existence of suitable selection criteria. All the solutions we introduce are solutions in the Kru\v zkov sense in the open half--planes $\reals_+\times \reals_+$  and  $\reals_+\times \reals_-$, so they satisfy the basic requirement of entropy dissipation. However, coming back to the real situation which our model aims to describe, namely the evacuation of a narrow corridor, we argue that the most desirable solution is obviously the one corresponding to the highest admissible values of the flux at the exit. In analogy to the discussion in Ref.~\cite{Libbano} we interpret all other possible solutions as consequences of an irrational behavior, which in literature is often described as \emph{panic}. It is also important to remark that since non--uniqueness is possible only when  $p(\rho_L-) \ne p(\rho_L+)$ and $p(\rho_L-) \ge f(\rho_L) \ge p(\rho_L+)$, non--uniqueness concerns at most a finite number of left states.

From now on we restrict ourselves to the case in which $p(\xi_i)$ can only take the values $p_i$ and $p_{i+1}$ and not the intermediate values, because the extremal behaviors are the most relevant in view of the applications.

\begin{definition}\label{def:Rp}
    Two Riemann solvers $\mathcal{R}^q$ and $\mathcal{R}^p$ for~\eqref{eq:constrianedRiemann} are defined as follows for $t>0$ sufficiently small and $x \in \reals$:
    \begin{description}
    \item[(C)] If $(\rho_L, \rho_R) \in \mathcal{C}$ then
      \[\mathcal{R}^q[\rho_L, \rho_R] (t,x)=\mathcal{R}^p[\rho_L, \rho_R](t,x) = \mathcal{R}[\rho_L, \rho_R](x/t).\]
    \item[(N)] If $(\rho_L, \rho_R) \in \mathcal{N}$ then
        $$\mathcal{R}^q[\rho_L, \rho_R] (t,x)= \mathcal{R}^p[\rho_L, \rho_R](t,x) = \left\{\begin{array}{l@{\quad\hbox{if }}l} \mathcal{R}[\rho_L, \hat\rho\left(\bar p \right)](x/t)&x<0\\ \mathcal{R}[\check\rho\left( \bar p \right), \rho_R](x/t)&x\ge0~,\end{array}\right.$$
      where $\bar p = p(\rho_L-)$ if $(\rho_L,\rho_R)$ satisfies~\textbf{(N4a)} or~\textbf{(N5a)}, otherwise $\bar p = p(\rho_L+)$.
    \item[(CN2), (CN3), (CNN5)] If $(\rho_L, \rho_R)$ satisfies one of these sets of conditions  then
        \begin{align*}
            \mathcal{R}^q[\rho_L, \rho_R](t,x) &= \mathcal{R}[\rho_L, \rho_R](x/t)~,\\
            \mathcal{R}^p[\rho_L, \rho_R](t,x) &= \left\{\begin{array}{l@{\quad\hbox{if }}l} \mathcal{R}[\rho_L, \hat\rho\left(p(\rho_L+)\right)](x/t)&x<0\\ \mathcal{R}[\check\rho\left(p(\rho_L+) \right), \rho_R](x/t)&x\ge0.\end{array}\right.
        \end{align*}
     \item[(NNN4), (NNN5)] If $(\rho_L, \rho_R)$ satisfies one of these sets of conditions then  $\mathcal{R}^q[\rho_L, \rho_R](t,x) $ takes the form \eqref{eq:nonclassicalsol} with $\bar p = p(\rho_L-) $ and $\mathcal{R}^p[\rho_L, \rho_R](t,x) $ takes the form \eqref{eq:nonclassicalsol} with $\bar p = p(\rho_L+) $.
    \end{description}
    \end{definition}
In the next proposition we collect the main properties of the Riemann solvers $\mathcal{R}^q$ and $\mathcal{R}^p$. In particular \textbf{(R6)} means that the Riemann solver $\mathcal{R}^q$ is the one which allows for the fastest evacuation, while $\mathcal{R}^p$ is associated to the slowest one.
\begin{proposition}\label{prop:Riemann}
Let $(\rho_L, \rho_R) \in \left[0,R\right]^2$. Then, for $\star = q,\,p$:
\begin{enumerate}
\item[\textbf{(R1)}] $\left[(t,x) \mapsto \mathcal{R}^\star[\rho_L, \rho_R](t,x) \right]$ is a weak solution to~\eqref{eq:constrianedRiemann1}, \eqref{eq:constrianedRiemann3}.
\item[\textbf{(R2)}] $\mathcal{R}^\star[\rho_L, \rho_R]$ satisfies the constraint~\eqref{eq:constrianedRiemann2} in the sense that
    \begin{align*}
        f\left(\mathcal{R}^\star[\rho_L, \rho_R](t,0\pm)\right) &\le p\left( \int_{\reals_-} w(x) ~ \mathcal{R}^\star[\rho_L, \rho_R] \left(t,x\right) ~{\d} x\right) .
    \end{align*}
\item[\textbf{(R3)}] $\mathcal{R}^\star[\rho_L, \rho_R](t) \in \BV\left( \reals;[0,R] \right)$.
\item[\textbf{(R4)}] The map $\mathcal{R}^\star ~\colon~ [0,R]^2 \to \Lloc1(\reals_+ \times \reals; \reals)$ is continuous in $\mathcal{C}\cup\mathcal{N}$ but not in all $[0,R]^2$.
\item[\textbf{(R5)}] $\mathcal{R}^\star$ is consistent, see Refs.~\cite{ColomboGoatinConstraint}, \cite{ColomboPriuli}.
\item[\textbf{(R6)}] $\mathcal{R}^q[\rho_L,\rho_R]$ maximizes the flux at the exit, in the sense that if $\mathcal{E}$ is the set of all entropy weak solutions of the Riemann problem \eqref{eq:constrianedRiemann}, we have
  $$
  \max_{\rho \in \mathcal{E}}\left\{f(\rho(t,0\pm)) \right\} = f\left(\mathcal{R}^q[\rho_L,\rho_R](0\pm)\right).
  $$
  Analogously, $\mathcal{R}^p[\rho_L,\rho_R]$ minimizes the flux at the exit, in the sense that
  $$
  \min_{\rho \in \mathcal{E}}\left\{f(\rho(t,0\pm)) \right\} = f\left(\mathcal{R}^p[\rho_L,\rho_R](0\pm)\right).
  $$
\end{enumerate}
\end{proposition}
\noindent The proof of Proposition~\ref{prop:Riemann} is deferred to Sec.~\ref{sec:technicalRiemann}.

It is important to observe that even if $p(\xi_i)$ can only take the two values $p_i$ and $p_{i+1}$, this is not enough to rule out the existence of  infinitely many different solutions as the ones described in Example~\ref{ex:nonuniqueness1}, in the case $p_i > f(\xi_i) = p_{i+1}$. However, the Riemann solver $\mathcal{R}^\star$ spontaneously selects one of them because it sticks to the constant level of constraint prescribed by Definition~\ref{def:Rp} until a nonlocal interaction takes place.
\begin{remark}\label{rem:catastrofe}
    Although the Riemann solvers $\mathcal{R}^\star$ are not $\Lloc1$--continuous, an existence result for the Cauchy problem~\eqref{eq:constrianed} can be obtained from a wave--front tracking algorithm based on $\mathcal{R}^\star$, see for instance Ref.~\cite{ColomboRosini2}. Such approach using $\mathcal{R}^\star$ does not require the operator splitting method. However, the non--local nature of the approximating problems prevents us from a direct application of the Riemann solvers $\mathcal{R}^\star$. In fact, even in a arbitrary small neighborhood of $x=0$, to prolong the approximating solution $\rho^n$ beyond a time $t = \bar t>0$ it is not sufficient to know the traces $\rho^n(\bar t,0-)$, $\rho^n(\bar t,0+)$, but also the value $\int_{-\iw}^0 w(x) ~\rho^n(\bar t,x) ~{\d}x$ is needed. Roughly speaking, because of the non--local character of the constraint one cannot merely juxtapose the solution to the Riemann problem associated to the values of the traces at $x=0$ with the solution to the Riemann problems away from the constraint. Finally, also jumps in $[t \mapsto p\left(\xi(t)\right)]$ have to be considered as (nonlocal) interactions. Therefore, the approach using $\mathcal{R}^\star$ is considerably heavier and more technical than the one we presented in Sec.~\ref{sec:Thealgorithm}, and we do not pursue this line in this paper.
\end{remark}

\section{Numerical examples}\label{sec:example}

 \begin{figure}
 \centering
 \subfigure[The solution in the $(x,t)$--plane.]
   {
      \centering
        \includegraphics[height=0.6\textheight]{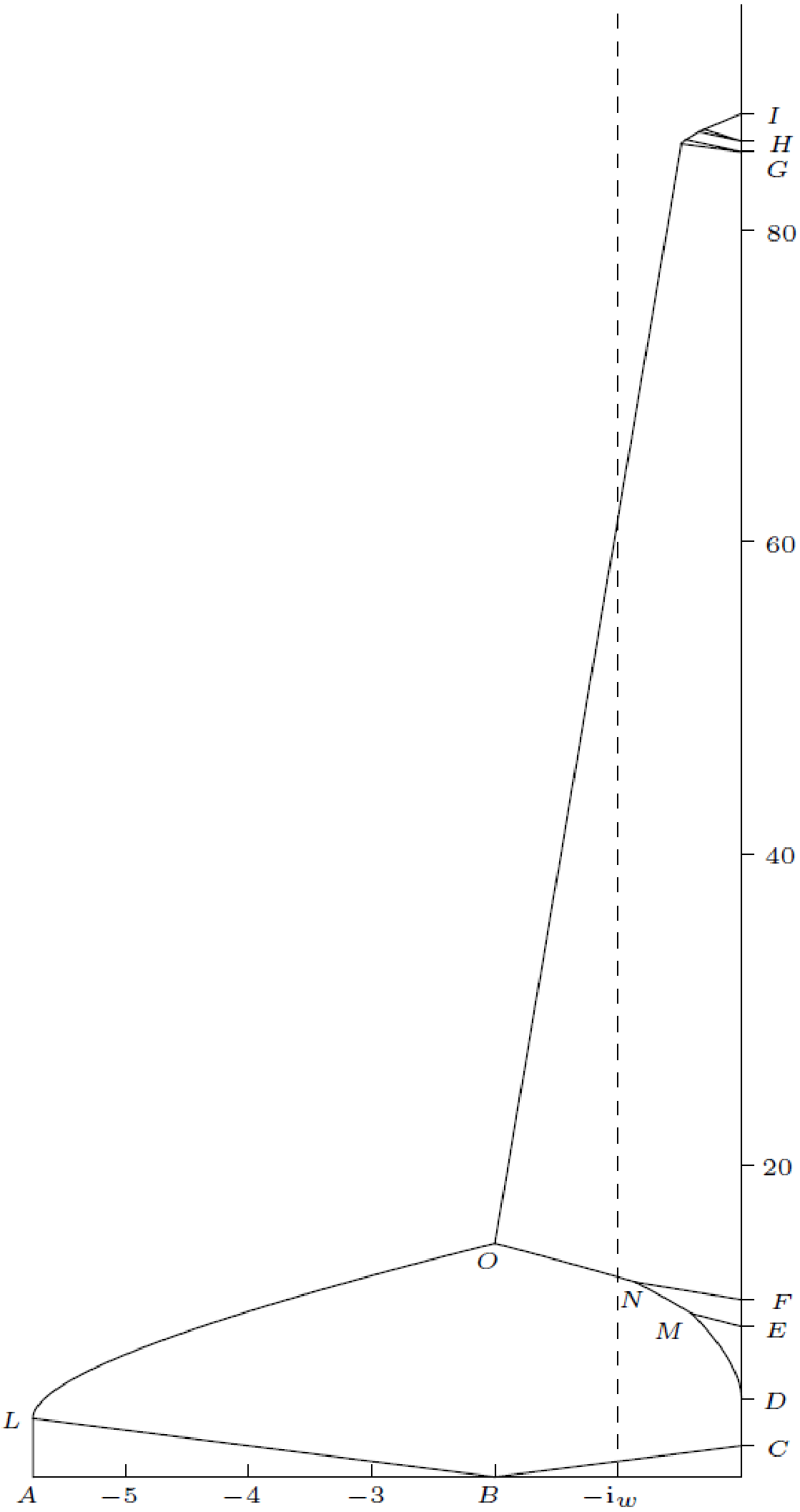}
}
 \hspace{5mm}
 \subfigure[Profiles of ${[x\mapsto\rho(t,x)]}$ at times $t = 0$, $t_E/2$, $(t_F+t_G)/2$, $(t_G+t_H)/2$, $(t_Q+t_R)/2$.]
   {
     \centering
     \includegraphics[height=0.6\textheight]{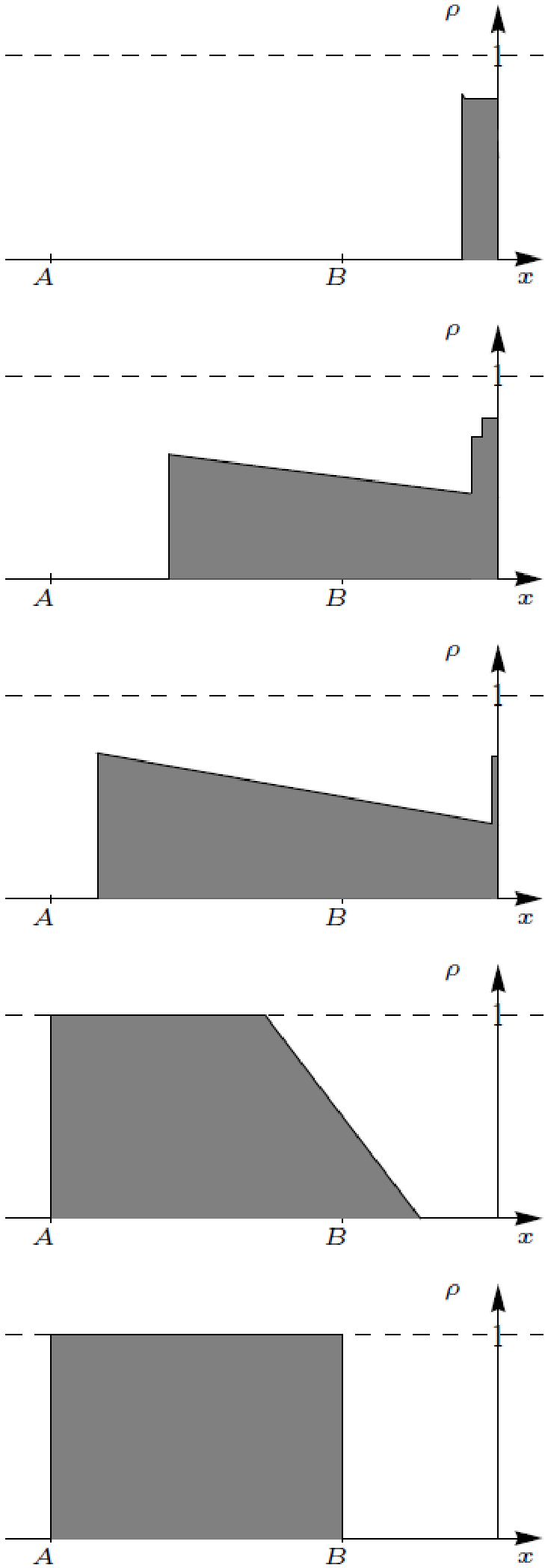}
}
       
 \caption{The solution described in Sec.~\ref{sec:example} and corresponding to the choice~\eqref{eq:numchoice}}
\label{fig:TheFlowerKings}
 \end{figure}

In this section we apply the model~\eqref{eq:constrianed} to simulate the evacuation of a corridor through an exit placed in $x=0$. The simulation is obtained by explicit analysis of the wave front interactions, with computer--assisted computation of front slopes and interaction times presented on Fig.~\ref{fig:TheFlowerKings}.

Assume that the pedestrians are initially uniformly distributed in $x \in [x_A, x_B]$ with maximal density, namely $\rho_0 = R~\chi_{[x_A, x_B]}$. As in Fig.~\ref{fig:TheFlowerKings2}~(a), we choose the efficiency of the exit, $p$, of the form
\begin{align*}
    p(\xi) &= \left\{
    \begin{array}{l@{\quad\hbox{ if }}l}
       p_0 & 0\le\xi<\xi_1\\
       p_1 & \xi_1\le\xi<\xi_2\\
       p_2 & \xi_2\le\xi\le R~,
    \end{array}
    \right.
\end{align*}
and such that the solution to each Riemann problem is unique and $\mathcal{R}^p \equiv \mathcal{R}^q$.

Then we can start with the construction of the solution. From $B=(x_B,0)$ starts the rarefaction $\mathcal{R}_B$ that takes the values $\mathcal{R}_B(t,x)$ implicitly given by
\begin{align*}
    \lambda(R) \le \lambda\left( \mathcal{R}_B(t,x) \right) = \frac{x-x_B}{t} \le \lambda(0)~.
\end{align*}
The first pedestrian reaches the exit at time $t_C= -x_B/\lambda(0)$. In $L=(x_A,t_L)$, with $t_L = (x_A-x_B)/\lambda(R)$, the stationary shock $\mathcal{C}_A$ originated from $A=(x_A,0)$ starts to interact with the rarefaction $\mathcal{R}_B$. As a result, from $L$ starts a shock $\mathcal{C}_L$ given by
\begin{align*}
    \mathcal{C}_L&\colon&
    \dot x(t) &= \sigma\left(0, \mathcal{R}_B\left(t,x(t)\right)\right) ~,&
    x(t_L)&=x_A.
\end{align*}
At time $t=t_D$ the maximal efficiency of the exit $p_0$ is reached, $f\left(\mathcal{R}_B(t_D,0)\right) = p_0$, and a queue appears behind it. The tail of the queue is represented by the backward shock $\mathcal{C}_D$ given by
\begin{align*}
    \mathcal{C}_D&\colon&
    \dot x(t) &= \sigma\left(\mathcal{R}_B\left(t,x(t)\right),\hat\rho(p_0)\right) ~,&
    x(t_D)&=0.
\end{align*}
Let $t_E$ be the value of $t$ solving the equation
\begin{align*}
    E&\colon&&
    \int_{-\iw}^{\mathcal{C}_D(t)} w(x) ~\mathcal{R}_B(t,x) ~{\d} x
    +\hat\rho(p_0) \int_{\mathcal{C}_D(t)}^0 w(x) ~{\d} x = \xi_1,\\
    &&&t>t_D~, \qquad \mathcal{C}_D(t) > -\iw ~.
\end{align*}
The data can be chosen in such a way that $\mathcal{C}_L(t_E) < -\iw$, then at time $t = t_E$ the efficiency of the exit falls to $p_1$ and a further shock $\mathcal{C}_E$ with constant speed $\sigma\left(\hat\rho(p_0), \hat\rho(p_1)\right)<0$ appears and reaches $\mathcal{C}_D$ in $M$. As a result, from $M$ starts the backward shock $\mathcal{C}_M$ given by
\begin{align*}
    \mathcal{C}_M&\colon&
    \dot x(t) &= \sigma\left(\mathcal{R}_B\left(t,x(t)\right),\hat\rho(p_1)\right) ~,&
    x(t_M)&=x_M.
\end{align*}
If $t_F$ is the solution of
\begin{align*}
    F&\colon&&
    \int_{-\iw}^{\mathcal{C}_M(t)} w(x) ~\mathcal{R}_B(t,x) ~{\d} x
    +\hat\rho(p_1) \int_{\mathcal{C}_M(t)}^0 w(x) ~{\d} x = \xi_2,\\
    &&&t>t_M~, \qquad \mathcal{C}_M(t) > -\iw ~,
\end{align*}
with $\mathcal{C}_L(t_F) < -\iw$, then the fall in the efficiency of the exit to $p_2$ affects the flow and from $F$ starts a shock $\mathcal{C}_F$ with constant speed $\sigma\left(\hat\rho(p_1), \hat\rho(p_2)\right)<0$ that reaches $\mathcal{C}_M$ in $N$. From $N$ then starts the backward shock $\mathcal{C}_N$ given by
\begin{align*}
    \mathcal{C}_N&\colon&
    \dot x(t) &= \sigma\left(\mathcal{R}_B\left(t,x(t)\right),\hat\rho(p_2)\right) ~,&
    x(t_N)&=x_N.
\end{align*}
We assume that $\mathcal{C}_N$ and $\mathcal{C}_L$ meet in $O$ with $x_O<-\iw$. Then from $O$ starts a forward shock $\mathcal{C}_O$. Observe that $\xi(t) = \hat\rho(p_2)$ for any time $t$ between $t_O$ and the time at which $\mathcal{C}_O$ crosses $x=-\iw$, see Fig.~\ref{fig:TheFlowerKings2}~(b). After that, the map $[t \mapsto \xi(t)]$ starts to decrease and, consequently, the efficiency of the exit increases at time $t_G$ when $\xi(t_G) = \xi_2$ and then again at time $t_H$ when $\xi(t_H) = \xi_1$. Observe that the shock $\mathcal{C}_O$ moves faster after its interaction with the two rarefactions started from $G$ and $H$ and that it finally reaches $x=0$ at time $t_I$, that corresponds to the evacuation time.

\begin{figure}[htpb]\tiny
      \centering
\subfigure[The functions ${[\rho \mapsto f(\rho)]}$ and ${[\xi \mapsto p(\xi)]}$.]
        {\includegraphics[width=0.32\textwidth]{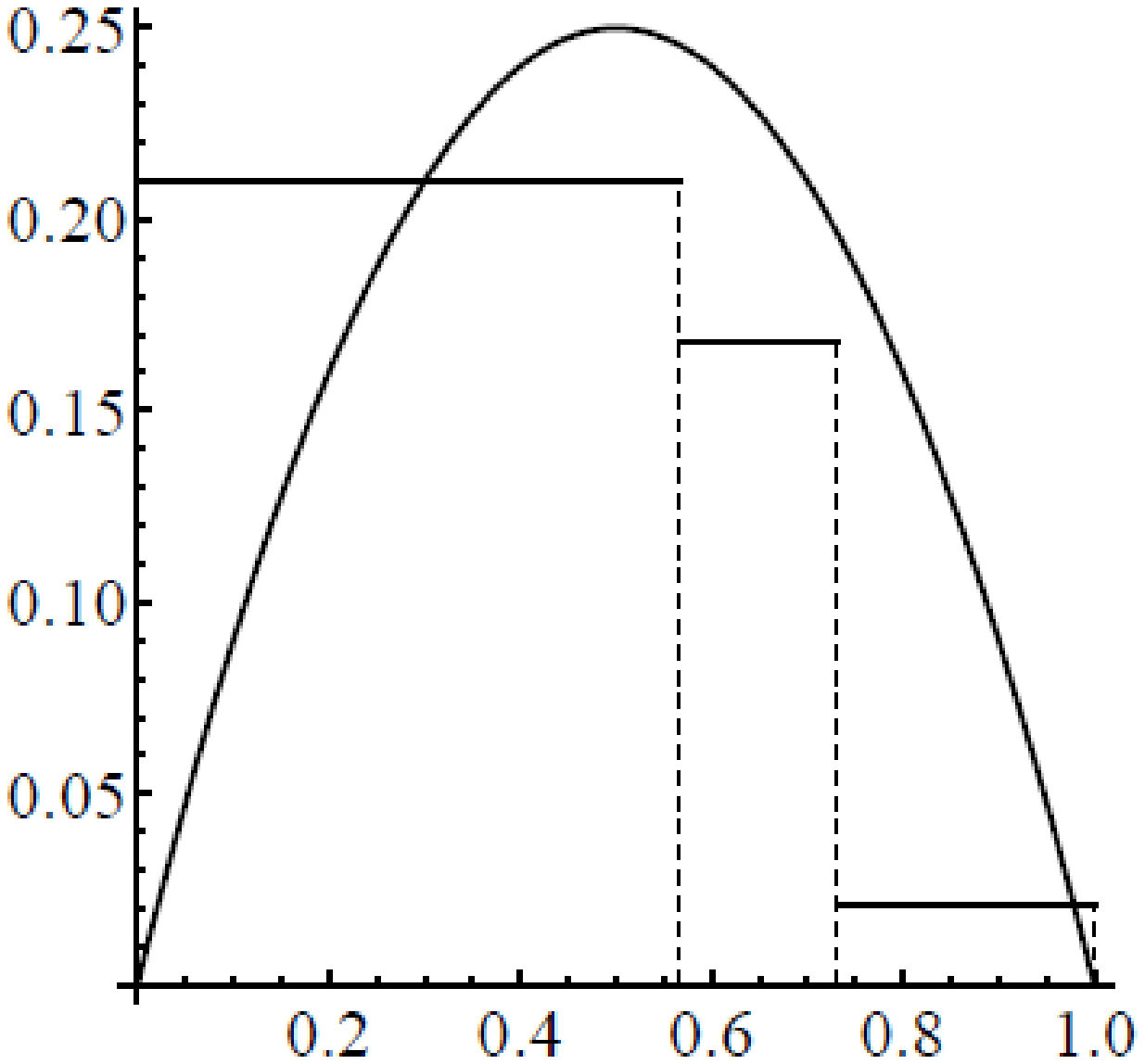}}
\subfigure[The function ${[t \mapsto \xi(t)]}$.]
        {\includegraphics[width=0.32\textwidth]{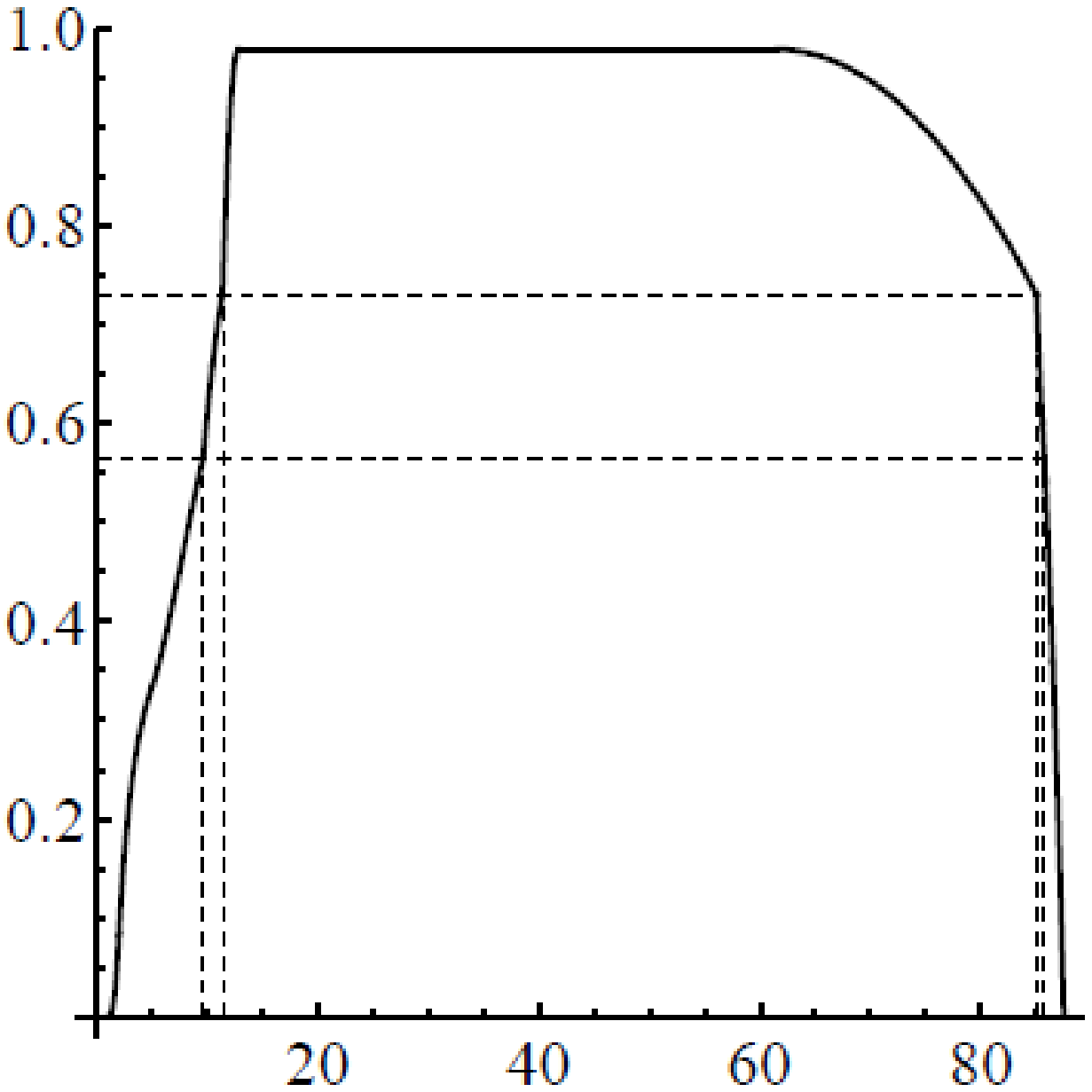}}
\subfigure[In gray the region $\mathcal{C}$ and in white the region $\mathcal{N}$.]
        {
        \includegraphics[width=0.33\textwidth]{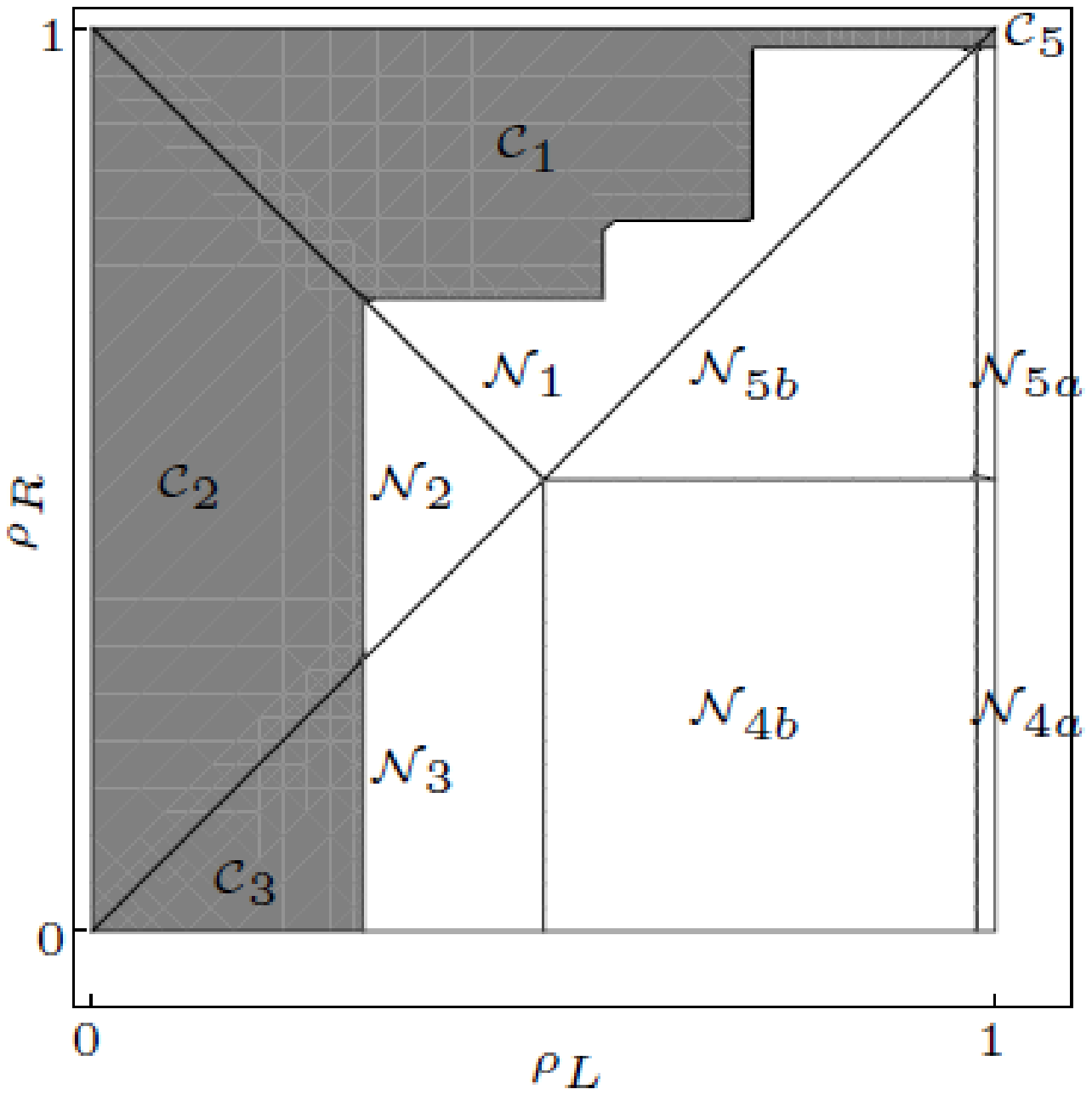}
        }
\caption{The above figures refer to Sec.~\ref{sec:example}.}
\label{fig:TheFlowerKings2}
\end{figure}

Fig.~\ref{fig:TheFlowerKings} corresponds to a linear weight function $w(x) = 2 \iw^{-2} (\iw + x)$, a normalized flux $f(\rho) = \rho(1-\rho)$ (namely the maximal velocity and the maximal density are assumed to be equal to one) and to the values
\begin{align}\nonumber
    p_0 &=0.21, & p_1 &=0.168, & p_2 &=0.021, & \xi_1 &\sim 0.566, & \xi_2 &\sim 0.731,\\ \nonumber
    x_A &= -5.75, & x_B &=-2, & \iw &=1, & x_O &=-2, & x_M &\sim-0.4002,\\ \label{eq:numchoice}
    t_C &=2, & t_D &=5, & t_E & \sim 9.651, & t_G &\sim85.045 ,& t_I &\sim 87.498.
\end{align}
In Fig.~\ref{fig:TheFlowerKings2}~(b), is represented the corresponding map $[t \mapsto \xi(t)]$. In Fig.~\ref{fig:TheFlowerKings2}~(c), we represent the region $\mathcal{C}$, in gray, and the region $\mathcal{N}$, in white, as introduced in Definition~\ref{def:Rp}. Notice that for this choice of $f$ and $p$ the region $\mathcal{C}_4$ happens to be empty.

\section{Technical section}\label{sec:tech}

\begin{lemma}\label{lem:star}
    Consider the family of scalar conservation laws
    \begin{align}\label{eq:sclappr}
        \partial_tu + \partial_xf^n(u) &= 0 & (t,x) &\in [0,T]\times\left]-\infty,0\right[ ~,
    \end{align}
    and assume that $f^n$ converges uniformly on compacts to $f$ as $n$ goes to infinity. Let $\rho^n$ be a sequence of Kru\v zkov entropy weak solutions to~\eqref{eq:sclappr}. If $\rho^n$ converges a.e.~to $\rho$ in $[0,T]\times\reals$, then $f^n\left(\rho^n(\cdot,0-)\right)$ converges weakly to $f\left(\rho(\cdot,0-)\right)$ in $\L1([0,T];\reals)$.
\end{lemma}
\begin{proof}
    First of all notice that by straightforward passage to the limit $\rho$ is a Kru\v zkov entropy weak solution to
    \begin{align}\label{eq:scl}
        \partial_tu + \partial_xf(u) &= 0 & (t,x) &\in [0,T]\times\left]-\infty,0\right[ ~.
    \end{align}
    Further, $\rho^n$ and $\rho$ admit strong left traces on $[0,T] \times\{0\}$, see Refs.~\cite{Panov}, \cite{Vasseur}.

    Now let $\delta_\varepsilon$ be as in~\eqref{eq:delta}. Choose any $\theta, \varphi \in \Cc\infty(\reals;\reals_+)$ such that $\theta(0) = \theta(T) = 0$, $\varphi(0) =1$ and observe that
    \begin{align*}
        \phi(t,x) &= \left[\int_{t-T+\varepsilon}^t \delta_\varepsilon(z) ~{\d}z\right] \left[\int^{x+1/\varepsilon}_{x+\varepsilon} \delta_\varepsilon(z) ~{\d}z\right]
        \theta(t) ~\varphi(x)
    \end{align*}
    is a $\Cc\infty$--map with support in $\left[0,T\right] \times \left[-1/\varepsilon,0\right]$ and, as $\varepsilon$ goes to zero
    \begin{align*}
        \phi(0,x)&\equiv0\equiv\phi(T,x)~, & \partial_t\phi(t,x) \to& \caratt{\left[0,T\right] \times \left]-\infty,0\right]}(t,x) ~\dot\theta(t) ~\varphi(x) ~,\\
        \phi(t,0)&\equiv0~, & \partial_x\phi(t,x) \to& \caratt{\left[0,T\right] \times \left]-\infty,0\right]}(t,x) ~\theta(t) ~\dot\varphi(x)\\
        &&&- \delta^D_{0-}(x) ~\caratt{[0,T]}(t) ~\theta(t) ~.
    \end{align*}
    By hypothesis $\rho^n$ is a weak solution of~\eqref{eq:sclappr}, therefore
    \begin{align*}
        \int_{0}^{T} \int_{\reals_-} \left[ \rho^n ~\partial_t\phi + f^n(\rho^n) ~\partial_x\phi \right] ~{\d} x ~{\d} t=0~,
    \end{align*}
    and letting $\varepsilon$ go to zero we have
    \begin{align*}
        \int_{0}^{T}\int_{\reals_-} \left[ \rho^n ~\varphi ~\dot\theta + f^n(\rho^n) ~\theta ~\dot\varphi \right] ~{\d}x ~{\d}t
        &=\int_{0}^{T} f^n(\rho^n(t,0-)) ~\theta(t) ~{\d} t ~.
    \end{align*}
    By definition $\rho$ is the strong limit in $\Lloc1$ of the sequence $\rho^n$ as $n$ goes to infinity and $f^n \to f$ uniformly on compacts. As $n$ goes to infinity, the left hand side of the above equation converges to
    \begin{align*}
        \int_{\reals_+} \int_{\reals_-} \left[ \rho ~\varphi ~\dot\theta + f(\rho) ~\theta ~\dot\varphi \right] ~{\d} x ~{\d} t ~.
    \end{align*}
    Arguing as above, we also find that
    \begin{align*}
        \int_{0}^{T}\int_{\reals_-} \left[ \rho ~\varphi ~\dot\theta + f(\rho) ~\theta ~\dot\varphi \right] ~{\d}x ~{\d}t
        &=\int_{0}^{T} f(\rho(t,0-)) ~\theta(t) ~{\d} t ~,
    \end{align*}
    because $\rho$ is a weak solution to~\eqref{eq:scl}. Therefore $\int_{0}^{T} f^n(\rho^n(t,0-)) ~\theta(t) ~{\d} t$ converges to $\int_{0}^{T} f(\rho(t,0-)) ~\theta(t) ~{\d} t$ as $n$ goes to infinity, and the weak limit of $f^n(\rho^n(t,0-))$ equals $f(\rho(t,0-))$.
\end{proof}

\subsection{Proof of Proposition~\ref{prop:interactionbound}}\label{sec:6.1}

    For any $\ell \in \naturals$ and $\bar t \in \left]0, \Delta t_h\right]$, we know by Ref.~\cite{ColomboGoatinConstraint} that the function $\rho^{n,h}_{\ell+1}(\bar t)$ is piecewise constant with jumps along a finite number of polygonal lines. Therefore, for any $t$ in a sufficiently small left neighborhood of $\bar t$, we can write, by~\eqref{eq:apprsol}
    \begin{align}\label{eq:Max}
        \rho^{n,h}(t+\ell\Delta t_h,x) &= \rho^{n,h}_{\ell+1}(t,x)
        = \sum_{i \in \mathcal{J}^{n,h}_{-}} \rho_{-,i}^{n,h} ~\caratt{\left[ s_{-,i-1}^{n,h}(t), s_{-,i}^{n,h}(t) \right[} (x)~,\\ \nonumber
        s_{-,i}^{n,h}(t) &= x_{-,i}^{n,h} + \sigma\left(\rho^{n,h}_{-,i}, \rho^{n,h}_{-,i+1}\right) ~\left(t-\bar t\right) ~,
    \end{align}
    where $\mathcal{J}^{n,h}_{-} \subset \integers$, $x_{-,0}^{n,h} = 0$, $\sigma\left(\rho^{n,h}_{-,0}, \rho^{n,h}_{-,1}\right) = 0$, $s_{-,0}^{n,h} \equiv 0$, $s_{-,i-1}^{n,h}(t) < s_{-,i}^{n,h}(t)$, $\rho_{-,i}^{n,h} \in \mathcal{M}^n$, $\rho_{-,i}^{n,h} \ne \rho_{-,i+1}^{n,h}$ for any $i\ne0$ and $f(\rho_{-,0}^{n,h}) = f(\rho_{-,1}^{n,h}) \le p^h(\Xi^{n,h}_{\ell-1})$. Introduce the notation
    \begin{eqnarray*}
        &f^{n,h}_{-,i} = f^n\left(\rho_{-,i}^{n,h}\right) ,
        \qquad
        \Psi_{-,i}^{n,h} = \Psi\left(\rho_{-,i}^{n,h}\right) ,\\
        &\hat\rho_+^{n,h} = \hat\rho\left( p^{n,h}_+ \right) ,
        \qquad
        \check\rho_+^{n,h} = \check\rho\left( p^{n,h}_+ \right) ,
        \qquad
        p_\pm^{n,h} = p^h\left( \Xi^{n,h}(\bar t\pm) \right) ,
        \\
        &\Delta\Upsilon^{n,h}_T\left(\bar t\right)
        =\Upsilon^{n,h}_T\left(\bar t+\right) - \Upsilon^{n,h}_T\left(\bar t-\right) .
    \end{eqnarray*}
    Observe that by definition:
    \begin{align}\label{eq:owio1}
        \hbox{if }\bar t < \Delta t_h\hbox{, then }&\quad p^{n,h}_- = p^{n,h}_+;\\
        \label{eq:owio2}
        &\quad f^{n,h}_{-,0} = f^{n,h}_{-,1} \le p^{n,h}_-;\\
        \label{eq:owio3}
        \hbox{if }\rho^{n,h}_{-,1} < \rho^{n,h}_{-,0}\hbox{, then }&\quad f^{n,h}_{-,0} = f^{n,h}_{-,1} = p^{n,h}_-.
    \end{align}
    We have to distinguish the following two main cases:
    \begin{enumerate}[\textbf{(U)}]
      \item[\textbf{(U)}] either $\bar t = \Delta t_h$, and in this case $\Gamma^{h}(\bar t+) - \Gamma^{h}(\bar t-) = - 5 \cdot2^{-h} ~f(\bar\rho)$;
      \item[\textbf{(I)}] or $\bar t \neq \Delta t_h$, and in this case $\Gamma^{h}(\bar t+) = \Gamma^{h}(\bar t-)$ and $p^{n,h}_- = p^{n,h}_+$ by~\eqref{eq:owio1}.
    \end{enumerate}
    Let us consider more in detail the case $\bar t = \Delta t_h$. If the efficiency of the exit does not change, \textit{i.e.}~$p^{n,h}_- = p^{n,h}_+$, then by hypothesis~\textbf{H2}, $\rho^{n,h}(t)$ is still given by~\eqref{eq:Max} for $t$ lying in a sufficiently small right neighborhood of $(\ell+1) \Delta t_h$, and therefore $\Delta\Upsilon^{n,h}_T(t) = - 5 \cdot2^{-h} ~f(\bar\rho)$. If the efficiency of the exit changes, then $\modulo{p^{n,h}_+ - p^{n,h}_-} = 2^{-h} f(\bar\rho)$ by Lemma~\ref{lem:capraEcavoli} and we have to distinguish the following cases:
    \begin{enumerate}[\textbf{(U0a)}]
    \item[\textbf{(U1)}] If the efficiency of the exit grows, then $p^{n,h}_+ = p^{n,h}_- + 2^{-h} f(\bar\rho)$ and $f^{n,h}_{-,0} = f^{n,h}_{-,1} < p^{n,h}_+$ by~\eqref{eq:owio2}. There are two possibilities:
    \begin{figure}[htpb]
        \centering
        \includegraphics[width=\textwidth]{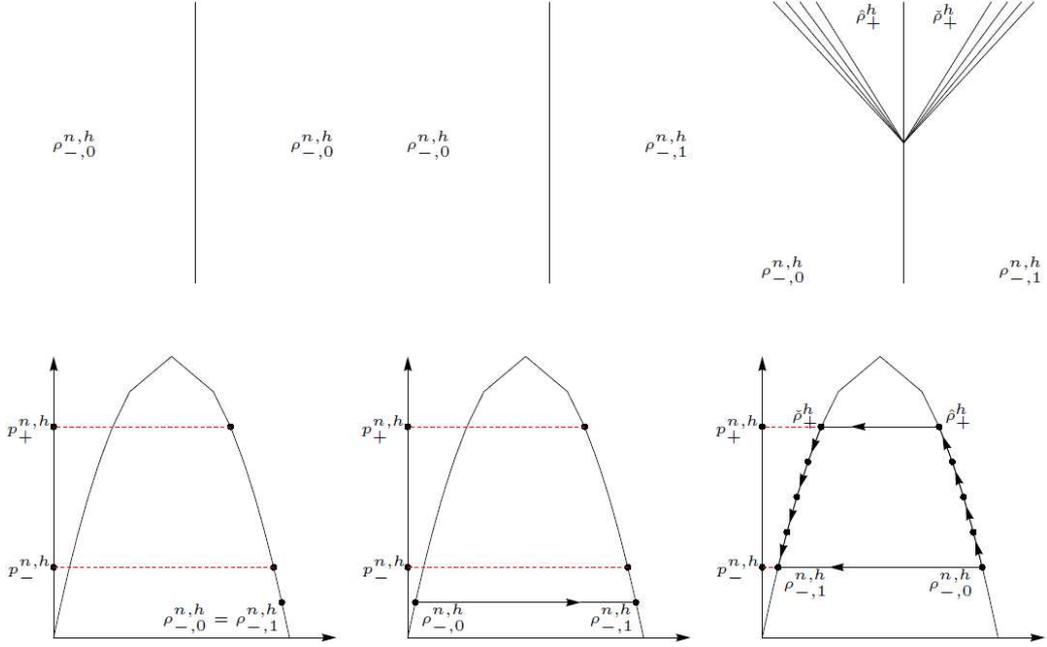}
        \caption{Interactions of the type~\textbf{(U1)}.}
        \label{fig:InterI1}
    \end{figure}
    \item[\textbf{(U1a)}] If $\rho^{n,h}_{-,0} \le \rho^{n,h}_{-,1}$, then the solution does not change its expression, see Fig.~\ref{fig:InterI1}, left and center, and
    \begin{align*}
        \Delta\Upsilon^{n,h}_T(\bar t)
        &= 4 \left[f(\bar\rho) - p^{n,h}_+\right] - 4 \left[f(\bar\rho) - p^{n,h}_-\right] - 5 \cdot2^{-h} f(\bar\rho)
        = - 9 \cdot2^{-h} f(\bar\rho).
    \end{align*}
    \item[\textbf{(U1b)}] If $\rho^{n,h}_{-,1} < \rho^{n,h}_{-,0}$, then $f^{n,h}_{-,0} = f^{n,h}_{-,1} = p^{n,h}_-$ by~\eqref{eq:owio3}. For $t>\bar t$ sufficiently small, the solution contains a rarefaction between $\rho^{n,h}_{-,0}$ and $\hat\rho_+^h$ on the left of the constraint, a nonclassical shock between $\hat\rho_+^h$ and $\check\rho_+^h$ at the constraint and a rarefaction between $\check\rho_+^h$ and $\rho^{n,h}_{-,1}$ on the right of the constraint, see Fig.~\ref{fig:InterI1}, right. Therefore
    \begin{align*}
        \Delta\Upsilon^{n,h}_T(\bar t)
        =& \left[p^{n,h}_+ - p^{n,h}_-\right] + 2\left[f(\bar\rho) - p^{n,h}_+\right] + \left[p^{n,h}_+ - p^{n,h}_-\right]\\
        &- 2\left[f(\bar\rho) - p^{n,h}_-\right] - 5 \cdot2^{-h} ~f(\bar\rho)
        = -5 \cdot2^{-h} ~f(\bar\rho) ~.
    \end{align*}
    \item[\textbf{(U2)}] If the efficiency of the exit decreases, then $p^{n,h}_+ = p^{n,h}_- - 2^{-h} f(\bar\rho)$. Differently from the case~\textbf{(U1)}, we do not know \textit{a priori} whether $p^{n,h}_+$ is less than $f^{n,h}_{-,0} = f^{n,h}_{-,1}$ or not. Therefore there are three possibilities:

    \begin{figure}[htpb]
        \centering
        \includegraphics[width=\textwidth]{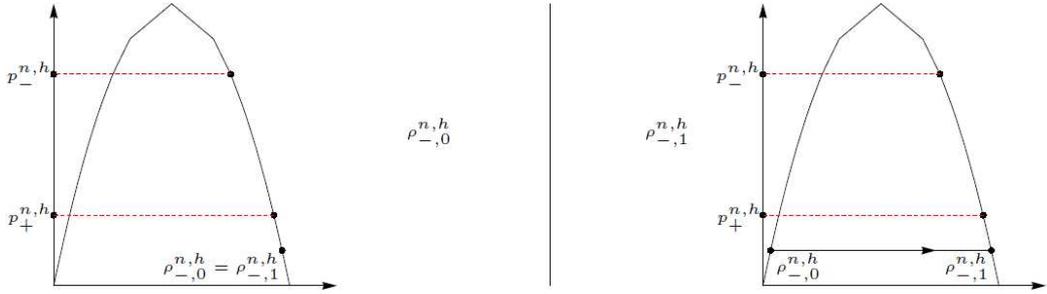}
        \caption{Interactions of the type~\textbf{(U2a)}.}
        \label{fig:InterI21}
    \end{figure}
    \begin{figure}[htpb]
        \centering
        \includegraphics[width=\textwidth]{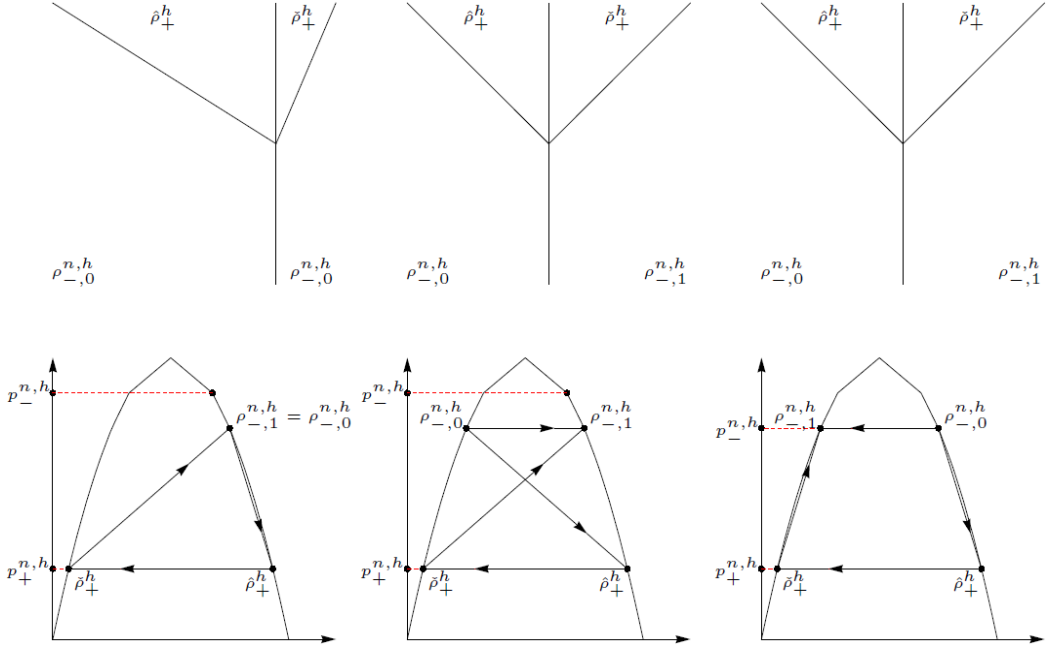}
        \caption{Interactions of the type~\textbf{(U2b)}, \textbf{(U2c)}.}
        \label{fig:InterI22}
    \end{figure}
    \item[\textbf{(U2a)}] If $\rho^{n,h}_{-,0} \le \rho^{n,h}_{-,1}$ and $f^{n,h}_{-,0} = f^{n,h}_{-,1} \le p^{n,h}_+$, then the solution does not change its expression, see Fig.~\ref{fig:InterI21}, and
    \begin{align*}
        \Delta\Upsilon^{n,h}_T(\bar t)
        &= 4 \left[f(\bar\rho) - p^{n,h}_+\right] - 4 \left[f(\bar\rho) - p^{n,h}_-\right] -5 \cdot2^{-h} ~f(\bar\rho)
        = -2^{-h} f(\bar\rho) ~.
    \end{align*}
    \item[\textbf{(U2b)}] If $\rho^{n,h}_{-,0} \le \rho^{n,h}_{-,1}$ and $p^{n,h}_+ < f^{n,h}_{-,0} = f^{n,h}_{-,1}$, then after time $\bar t$ the solution performs a shock between $\rho^{n,h}_{-,0}$ and $\hat\rho_+^h$ on the left of the constraint, a nonclassical shock between $\hat\rho_+^h$ and $\check\rho_+^h$ at the constraint and a shock between $\check\rho_+^h$ and $\rho^{n,h}_{-,1}$ on the right of the constraint, see Fig.~\ref{fig:InterI22}, left and center. Therefore, if $\rho^{n,h}_{-,0} \ne \rho^{n,h}_{-,1}$, then
    \begin{align*}
        \Delta\Upsilon^{n,h}_T(\bar t)
        =& 2\left[2f(\bar\rho) - f^{n,h}_{-,1} - p^{n,h}_+\right] + 2 \left[f(\bar\rho) - p^{n,h}_+\right] - 2 \left[f(\bar\rho) - f^{n,h}_{-,1}\right]\\
        & - 4 \left[f(\bar\rho) - p^{n,h}_-\right] -5 \cdot2^{-h} ~f(\bar\rho)
        = -2^{-h} f(\bar\rho)~,
    \end{align*}
    while, if $\rho^{n,h}_{-,0} = \rho^{n,h}_{-,1}$, then
    \begin{align*}
        \Delta\Upsilon^{n,h}_T(\bar t)
        =& \left[2f(\bar\rho) - f^{n,h}_{-,1} - p^{n,h}_+\right] + 2 \left[f(\bar\rho) - p^{n,h}_+\right] + \left[f^{n,h}_{-,1} - p^{n,h}_+\right]\\
        &- 4 \left[f(\bar\rho) - p^{n,h}_-\right] -5 \cdot2^{-h} ~f(\bar\rho)
        = -2^{-h} f(\bar\rho)~.
    \end{align*}
    \item[\textbf{(U2c)}] If $\rho^{n,h}_{-,1} < \rho^{n,h}_{-,0}$, then $p^{n,h}_+ < p^{n,h}_- = f^{n,h}_{-,0} = f^{n,h}_{-,1}$ by~\eqref{eq:owio3}, and after time $\bar t$ the solution performs a shock between $\rho^{n,h}_{-,0}$ and $\hat\rho_+^h$ on the left of the constraint, a nonclassical shock between $\hat\rho_+^h$ and $\check\rho_+^h$ at the constraint and a shock between $\check\rho_+^h$ and $\rho^{n,h}_{-,1}$ on the right of the constraint, see Fig.~\ref{fig:InterI22}, right. Therefore
    \begin{align*}
        \Delta\Upsilon^{n,h}_T(\bar t)
        =& \left[p^{n,h}_- - p^{n,h}_+\right] + 2\left[f(\bar\rho) - p^{n,h}_+\right] + \left[p^{n,h}_- - p^{n,h}_+\right]\\
        &- 2\left[f(\bar\rho) - p^{n,h}_-\right] -5 \cdot2^{-h} ~f(\bar\rho)
        = -2^{-h} f(\bar\rho)~.
    \end{align*}
    \end{enumerate}
    In conclusion, for the case~\textbf{(U)} we proved that $\Delta\Upsilon^{n,h}_T\left( (\ell+1) \Delta t_h\right) \le -2^{-h} f(\bar\rho) \le -2^{-n} f(\bar\rho)$ for any $\ell \in \naturals$.

    Now, assume that we have $\bar t \in \left]0, \Delta t_h\right[$. If at time $\bar t$ no interaction occurs, then $\Delta\Upsilon^{n,h}_T(\bar t) = 0$. The remaining part of the proof consists in a detailed study of all possible interactions. We start with the most classical case when the interaction occurs away from $x=0$.
    \begin{enumerate}[\textbf{(I0a)}]
    \item[\textbf{(I0)}]If two waves, respectively between $\rho^{n,h}_{i-1}$ and $\rho^{n,h}_i$ and between $\rho^{n,h}_{i}$ and $\rho^{n,h}_{i+1}$, interact away from the constraint, then at least one of the two waves has to be a shock and, in any case, the resulting wave is a shock between $\rho^{n,h}_{i-1}$ and $\rho^{n,h}_{i+1}$. Therefore
        \begin{align*}
            \Delta\Upsilon^{n,h}_T(\bar t)
            &=\modulo{\Psi^{n,h}_{-,i-1} - \Psi^{n,h}_{-,i+1}} - \modulo{\Psi^{n,h}_{-,i-1} - \Psi^{n,h}_{-,i}} - \modulo{\Psi^{n,h}_{-,i} - \Psi^{n,h}_{-,i+1}} \le 0
        \end{align*}
        and the number of waves after the interaction diminishes.
    \end{enumerate}
    We now study the case in which a rarefaction reaches $x=0$.
    \begin{enumerate}[\textbf{(I1c)}]
    \item[\textbf{(I1)}] If a rarefaction reaches $x=0$ from the left, then $\rho^{n,h}_{-,0} < \rho^{n,h}_{-,-1} \le \bar\rho$ and $f^{n,h}_{-,0} + 2^{-n} = f^{n,h}_{-,-1}$. In particular, the solution cannot perform a nonclassical shock at $x=0$ before the interaction. There are therefore three possibilities:
    \item[\textbf{(I1a)}] If $\rho^{n,h}_{-,0} = \rho^{n,h}_{-,1}$ and $f^{n,h}_{-,-1} \le p^{n,h}_\pm$, then the rarefaction crosses $x=0$ and $\Delta\Upsilon^{n,h}_T(\bar t) = 0$.
    \item[\textbf{(I1b)}] If $\rho^{n,h}_{-,0} = \rho^{n,h}_{-,1}$ and $p^{n,h}_\pm < f^{n,h}_{-,-1}$, then $p^{n,h}_\pm = f^{n,h}_{-,0} = f^{n,h}_{-,1}$ because $n>h$ and for~\eqref{eq:owio2}. After the interaction, the solution performs a shock between $\rho^{n,h}_{-,-1}$ and $\hat\rho_+^h$ on the left of the constraint and a nonclassical shock between $\hat\rho_+^h$ and $\rho^{n,h}_{-,1}$ at the constraint. Therefore
    \begin{align*}
        \Delta\Upsilon^{n,h}_T(\bar t) &
        = \left[2f(\bar\rho) - f^{n,h}_{-,-1} - p^{n,h}_\pm\right] + 2 \left[f(\bar\rho) - p^{n,h}_\pm\right]- \left[f^{n,h}_{-,-1} - p^{n,h}_\pm\right]\\
        & - 4 \left[f(\bar\rho) - p^{n,h}_\pm\right] = 2\left[ p^{n,h}_\pm - f^{n,h}_{-,-1}\right] = -2^{1-n}f(\bar\rho)~.
    \end{align*}
    \item[\textbf{(I1c)}] If $\rho^{n,h}_{-,0} \ne \rho^{n,h}_{-,1}$, then after the interaction the solution performs a shock between $\rho^{n,h}_{-,-1}$ and $\rho^{n,h}_{-,1}$ on the left of the constraint. Therefore
    \begin{align*}
        \Delta\Upsilon^{n,h}_T(\bar t)
        =& \left[2f(\bar\rho) - f^{n,h}_{-,-1} - f^{n,h}_{-,1}\right]- \left[f^{n,h}_{-,-1} - f^{n,h}_{-,1}\right] \\
        &- 2\left[f(\bar\rho) - f^{n,h}_{-,1}\right] = 2\left[ f^{n,h}_{-,1} - f^{n,h}_{-,-1}\right] = -2^{1-n}f(\bar\rho) ~.
    \end{align*}
    \item[\textbf{(I2)}] If a rarefaction reaches $x=0$ from the right, then the analysis turns out to be completely analogous to that for the case~\textbf{(I1)}, and therefore it is omitted.
    \end{enumerate}
    We now study the cases when a shock reaches $x=0$.
    \begin{enumerate}[\textbf{(I0a)}]
    \item[\textbf{(I3)}] If a shock reaches $x=0$ from the left, then $\rho^{n,h}_{-,-1} < \min\{\rho^{n,h}_{-,0}, \bar\rho\}$, $f^{n,h}_{-,-1} < f^{n,h}_{-,0} = f^{n,h}_{-,1} \le p^{n,h}_\pm$ by~\eqref{eq:owio2}. There are three possibilities:
    \item[\textbf{(I3a)}] If $\rho^{n,h}_{-,0} = \rho^{n,h}_{-,1}$, then the shock crosses $x=0$ and we have $\Delta\Upsilon^{n,h}_T(\bar t) = 0$.
    \item[\textbf{(I3b)}] If $\rho^{n,h}_{-,0} < \rho^{n,h}_{-,1}$, then after the interaction the solution performs a shock between $\rho^{n,h}_{-,-1}$ and $\rho^{n,h}_{-,1}$ on the right of the constraint. Therefore
    \begin{align*}
        \Delta\Upsilon^{n,h}_T(\bar t)
        =& \left[2f(\bar\rho) - f^{n,h}_{-,1} - f^{n,h}_{-,-1}\right] - \left[f^{n,h}_{-,1} - f^{n,h}_{-,-1}\right] - 2 \left[f(\bar\rho) - f^{n,h}_{-,1}\right] = 0 ~.
    \end{align*}
    \item[\textbf{(I3c)}] If $\rho^{n,h}_{-,1} < \rho^{n,h}_{-,0}$, then $p^{n,h}_\pm = f^{n,h}_{-,0} = f^{n,h}_{-,1}$ by~\eqref{eq:owio3}, and after the interaction the solution performs a shock between $\rho^{n,h}_{-,-1}$ and $\rho^{n,h}_{-,1}$ on the right of the constraint. Therefore
    \begin{align*}
        \Delta\Upsilon^{n,h}_T(\bar t)
        =& \left[p^{n,h}_\pm - f^{n,h}_{-,-1}\right] + 4 \left[f(\bar\rho) - p^{n,h}_\pm\right]\\
        &- \left[2f(\bar\rho) - p^{n,h}_\pm - f^{n,h}_{-,-1}\right] - 2 \left[f(\bar\rho) - p^{n,h}_\pm\right] = 0 ~.
    \end{align*}
    \item[\textbf{(I4)}] If a shock reaches $x=0$ from the right, then the analysis turns out to be completely analogous to that for the case~\textbf{(I3)}, and therefore it is omitted.
    \end{enumerate}
    This concludes the proof of Proposition~\ref{prop:interactionbound}.

\subsection{Proof Proposition~\ref{prop:lallelarsson}}\label{sec:lallelarsson}

    We have to prove that the conditions given in Definition~\ref{def:entropysol} are satisfied. Fix $\phi \in \Cc\infty(\reals^2; \reals_+)$ and $k \in [0,R]$. Let $\delta_\varepsilon$ be as in~\eqref{eq:delta}. Consider
    \begin{align*}
        \phi_\varepsilon(t,x) &= \sum_{\ell\in\naturals} \phi_\varepsilon^\ell(t-\ell\Delta t_h,x) = \sum_{\ell\in\naturals} \phi(t,x) \int_{t-(\ell+1)\Delta t_h+\varepsilon}^{t-\ell\Delta t_h} \delta_\varepsilon(y) ~{\d}y
    \end{align*}
    and observe that $\phi_\varepsilon, \phi_\varepsilon^\ell \in \Cc\infty(\reals^2; \reals_+)$, $\phi_\varepsilon^\ell(\cdot,x)$ has support in $[0, \Delta t_h]$ for any $x \in\reals$, and as $\varepsilon$ goes to zero
    \begin{align*}
        \phi_\varepsilon(t,x) &\to \phi(t,x)~,\\
        \partial_t\phi_\varepsilon(t,x) &\to \partial_t\phi(t,x) + \sum_{\ell\in\naturals} \phi(t,x) \left[\delta^D_{\ell\Delta t_h}(t) - \delta^D_{(\ell+1)\Delta t_h}(t)\right],\\
        \partial_x\phi_\varepsilon(t,x) &\to \partial_x\phi(t,x) ~.
    \end{align*}
    By construction, since each $[t \mapsto \rho^{n,h}_{\ell+1}(t)]$ is an entropy weak solution of~\eqref{eq:constrianedapprk} in the sense of Definition~\ref{def:entropysol}, we have
    \begin{align*}
        &\int_0^{\Delta t_h} \int_\reals \modulo{\rho^{n,h}_{\ell+1}(t,x)-k} \partial_t\phi_\varepsilon^\ell(t,x)  ~{\d} x ~{\d} t\\
        &+\int_0^{\Delta t_h} \int_\reals \sign(\rho^{n,h}_{\ell+1}(t,x)-k) \left[f^n\left(\rho^{n,h}_{\ell+1}(t,x)\right) - f^n(k)\right] \partial_x\phi_\varepsilon^\ell(t,x) ~{\d} x ~{\d} t\\
        &+ 2 \int_0^{\Delta t_h} \left[1 - \dfrac{p^h\left( \Xi^{n,h}_\ell(t) \right)}{f^n(\bar\rho)}   \right] f^n(k) ~\phi_\varepsilon^\ell(t,0) ~{\d} t \ge0 ~.
    \end{align*}
    By summing over $\ell$ and letting $\varepsilon$ go to zero, by~\eqref{eq:apprsol} we obtain that
    \begin{align*}
        0\le&\sum_{\ell\in\naturals} \Bigg\{\int_0^{\Delta t_h} \int_\reals \modulo{\rho^{n,h}_{\ell+1}(t,x)-k} \partial_t\phi(t+\ell\Delta t_h,x)  ~{\d} x ~{\d} t\\
        &+ \int_\reals \modulo{\rho^{n,h}_{\ell+1}(0,x)-k} ~\phi(\ell\Delta t_h,x)  ~{\d} x
        - \int_\reals \modulo{\rho^{n,h}_{\ell+1}(\Delta t_h,x)-k} \phi\left( (\ell+1)\Delta t_h,x \right)  {\d} x
        \\
        &+\int_0^{\Delta t_h}\!\!\! \int_\reals \sign(\rho^{n,h}_{\ell+1}(t,x)-k) \left[f^n\left(\rho^{n,h}_{\ell+1}(t,x)\right) - f^n(k)\right] \partial_x\phi\left(t+\ell\Delta t_h,x\right) {\d} x ~{\d} t\\
        &+ 2 \int_0^{\Delta t_h} \left[1 - \dfrac{p^h\left( \Xi^{n,h}_\ell(t) \right)}{f^n(\bar\rho)}   \right] f^n(k) ~\phi(t+\ell\Delta t_h,0) ~{\d} t \Bigg\}\\
        =&
        \int_{\reals_+} \int_\reals \modulo{\rho^{n,h}(t,x)-k} \partial_t\phi(t,x)  ~{\d} x ~{\d} t\\
        &+\int_{\reals_+} \int_\reals \sign(\rho^{n,h}(t,x)-k) \left[f^n\left(\rho^{n,h}(t,x)\right) - f^n(k)\right] \partial_x\phi(t,x) ~{\d} x ~{\d} t\\
        &+ 2 \int_{\reals_+} \left[1 - \dfrac{p^h\left( \Xi^{n,h}(t) \right)}{f^n(\bar\rho)}   \right] f^n(k) ~\phi(t,0) ~{\d} t
        + \int_{\reals_+} \modulo{\rho^{n}_0(x)-k} ~\phi(0,x) ~{\d}x ~.
    \end{align*}
    Finally, by construction, $f^n\left(\rho^{n,h}(t, 0\pm)\right) \le p^h\left( \Xi^{n,h}(t) \right)$ for a.e.~$t \in \reals_+$, and this ends the proof of Proposition~\ref{prop:lallelarsson}.

\subsection{Proof of Proposition~\ref{prop:riemann}}\label{sec:proofRiem}

We list here two basic properties which will be of great help in the following case by case analysis.

First, by definition~\eqref{eq:xi}, $\xi(0)=\rho_L$ and the map $[t \mapsto \xi(t)]$ is continuous. Thus, by assumption~\textbf{(P2)} we have that for any $t>0$ sufficiently small
    \begin{description}
      \item[bp1] if $\xi(t) < \rho_L$, then $p(\xi(t)) \equiv p(\rho_L-)$;
      \item[bp2] if $\rho_L < \xi(t)$, then $p(\xi(t)) \equiv p(\rho_L+)$.
    \end{description}
The case $\xi(t) \equiv \rho_L$ is somehow special and has to be studied separately for each specific case.

Second, when the solution is nonclassical, due to the finite speed of propagation of the waves,  the assumption~\textbf{(P2)} and properties~\textbf{bp1} and~\textbf{bp2}, we have
\begin{description}
  \item[np1] if $\mathcal{R}\left[\rho_L, \hat\rho(\bar p)\right](x) \equiv \rho_L$ for $x<0$, then $\bar p = f(\rho_L) \in \left[p(\rho_L+), p(\rho_L-)\right]$;
  \item[np2] if $\bar p \ne f(\rho_L)$ and $\rho_L < \hat\rho(\bar p)$, then $\bar p = p(\rho_L+)$;
  \item[np3] if $\bar p \ne f(\rho_L)$ and $\hat\rho(\bar p) < \rho_L$, then $\bar p = p(\rho_L-)$;
  \item[np4] if $p$ is continuous in $\rho_L$, namely $p(\rho_L-)=p(\rho_L+)$, then $\bar p = p(\rho_L)$.
\end{description}

Now we start the description of the possible cases and we proceed as following. First, we show that for any initial datum satisfying~\textbf{(C$i$)}, $i=1,\ldots,5$, the problem actually has a unique solution and that the solution is classical. Second, we take into consideration the corresponding case~\textbf{(N$i$)}, for which we prove that the classical solution is not suitable and that there exists a unique nonclassical solution.

It is important to stress that in general the solutions to the constrained Riemann problem~\eqref{eq:constrianedRiemann} are not self--similar, see Example~\ref{ex:es0}. All the cases listed below describe self--similar solutions because we let the solutions to evolve only on a small interval of time.

    \begin{enumerate}[\textbf{(N4a)}]
      \item[\textbf{(C1)}]  In this case $\left[ (t,x) \mapsto \mathcal{R}[\rho_L, \rho_R](x/t) \right]$ performs a shock with negative speed $\sigma(\rho_L, \rho_R)$ and satisfies~\eqref{eq:constrianedRiemann2} because $f(\rho_R) \le p(\rho_L+)$ and $p(\xi(t)) \equiv p(\rho_L+)$ by~\textbf{bp2}. Assume that there exists a nonclassical solution of the form~\eqref{eq:nonclassicalsol}. Observe that the assumptions $\rho_L < \rho_R$ and $f(\rho_R) < f(\rho_L)$ together imply that $\bar\rho < \rho_R$.  Then $\check\rho(\bar p) \le \bar\rho < \rho_R$ and $\mathcal{R}[\check\rho(\bar p),\rho_R]$ is given by a shock with non negative speed if and only if $\bar p \le f(\rho_R)$, or equivalently, $\rho_R \le \hat\rho(\bar p)$. As a consequence, $\bar p \le f(\rho_R) < f(\rho_L)$, $\rho_L < \rho_R \le \hat\rho(\bar p)$ and by~\textbf{np2} $\bar p$ coincides with $p(\rho_L+)$. In conclusion we have $\bar p \le f(\rho_R) \le p(\rho_L+) = \bar p$, namely $f(\rho_R) =\bar p$ and the nonclassical solution coincides with the classical one.
      \item[\textbf{(N1)}] In this case $\left[ (t,x) \mapsto \mathcal{R}[\rho_L, \rho_R](x/t) \right]$ does not satisfy~\eqref{eq:constrianedRiemann2} because $f(\rho_R) > p(\rho_L+)$, see case~\textbf{(C1)}. Therefore, there does not exist any classical solution and we can consider only nonclassical solutions of the form~\eqref{eq:nonclassicalsol}. If $p$ is continuous in $\rho_L$, then by~\textbf{np4} we have that $\bar p = p(\rho_L)$. If $p$ experiences a jump at $\rho_L$ then, one may wonder which value in $\left[p(\rho_L+), p(\rho_L-)\right]$ has to be chosen as $\bar p$. As in the case~\textbf{(C1)}, the assumptions  imply that $\bar\rho < \rho_R$ and then that $\check\rho(\bar p) < \rho_R$ and $\bar p \le f(\rho_R)$. Then $\bar p$ is strictly smaller than $f(\rho_L)$ and $\hat\rho(\bar p) > \rho_L$. As a consequence, property~\textbf{np2} forces us to choose the unique possible value of $\bar p$, which is $p(\rho_L+)$.
      \item[\textbf{(C2)}] In this case $\left[ (t,x) \mapsto \mathcal{R}[\rho_L, \rho_R](x/t) \right]$ performs a shock with non negative speed $\sigma(\rho_L, \rho_R)$ and it satisfies~\eqref{eq:constrianedRiemann2} because $f(\rho_L) \le p(\rho_L+)$. Assume that there exists a nonclassical solution of the form~\eqref{eq:nonclassicalsol}. Observe that the assumptions $\rho_L < \rho_R$ and $f(\rho_R) \ge f(\rho_L)$ together imply that $\bar\rho > \rho_L$.  Then $\hat\rho(\bar p) \ge \bar\rho >\rho_L$ and $\mathcal{R}[\rho_L, \hat\rho(\bar p)]$ is given by a shock with non positive speed if and only if $\bar p \le f(\rho_L)$. Thus $\bar p \le f(\rho_L) \le p(\rho_L+)$ and this implies by~\eqref{eq:nonclassicalsol3} that $\bar p = f(\rho_L) = p(\rho_L+)$ and that the nonclassical solution coincides with the classical one.
      \item[\textbf{(N2)}] In this case $\left[ (t,x) \mapsto \mathcal{R}[\rho_L, \rho_R](x/t) \right]$ does not satisfy~\eqref{eq:constrianedRiemann2} because $f(\rho_L) > p(\rho_L-)$, see case~\textbf{(C2)}. Therefore, there does not exist any classical solution and we can consider only nonclassical solutions of the form~\eqref{eq:nonclassicalsol}. As in the case~\textbf{(C2)},  the assumptions imply $\hat\rho(\bar p) \ge \bar\rho >\rho_L$. Furthermore, by~\eqref{eq:nonclassicalsol3} we have $\bar p \le p(\rho_L-) < f(\rho_L)$, and as a consequence, property~\textbf{np2} forces us to choose $\bar p = p(\rho_L+)$.
      \item[\textbf{(C3)}] In this case $\left[ (t,x) \mapsto \mathcal{R}[\rho_L, \rho_R](x/t) \right]$ performs a possible null rarefaction on the right of the constraint and it satisfies~\eqref{eq:constrianedRiemann2} because $f(\rho_L) \le p(\rho_L+)$. Assume that there exists a nonclassical solution of the form~\eqref{eq:nonclassicalsol}. Since $\rho_L \le \bar\rho \le \hat\rho(\bar p)$, $\mathcal{R}[\rho_L,\hat\rho(\bar p)]$ is given by a shock that has non positive speed if and only if $\bar p \le f(\rho_L)$. Therefore $\bar p \le f(\rho_L) \le p(\rho_L+)$ and this by~\eqref{eq:nonclassicalsol3} implies that $\bar p = f(\rho_L) = p(\rho_L+)$ and that the nonclassical solution coincides with the classical one.
      \item[\textbf{(N3)}] In this case $\left[ (t,x) \mapsto \mathcal{R}[\rho_L, \rho_R](x/t) \right]$ does not satisfy~\eqref{eq:constrianedRiemann2} because $f(\rho_L) > p(\rho_L-)$, see case~\textbf{(C3)}. Therefore, there does not exist any classical solution and we can consider only nonclassical solutions of the form~\eqref{eq:nonclassicalsol}. By hypothesis and~\eqref{eq:nonclassicalsol3} we have $f(\rho_L) > p(\rho_L-) \ge \bar p$. Therefore $\rho_L \le \bar\rho < \hat\rho(\bar p)$ and by~\textbf{np2} we have $\bar p = p(\rho_L+)$.
      \item[\textbf{(C4)}] In this case $\left[ (t,x) \mapsto \mathcal{R}[\rho_L, \rho_R](x/t) \right]$ performs a rarefaction with speeds between $\lambda(\rho_L) <0$ and $\lambda(\rho_R)\ge0$ and it satisfies~\eqref{eq:constrianedRiemann2} because $f(\bar\rho) = p(\rho_L+)$ implies that $p(\rho) = f(\bar\rho)$ for all $\rho \le \rho_L$. Moreover, it implies also that $p$ is continuous in $\rho_L$ and therefore, by~\textbf{np4}, any nonclassical solution of the form~\eqref{eq:nonclassicalsol} must have $\bar p = p(\rho_L) = f(\bar\rho)$, but in this case the nonclassical solution coincides with the classical one.
     \item[\textbf{(N4)}] In this case $\left[ (t,x) \mapsto \mathcal{R}[\rho_L, \rho_R](x/t) \right]$ does not satisfy~\eqref{eq:constrianedRiemann2} because $f(\bar\rho) > p(\rho_L-)$, see case~\textbf{(C4)}. Therefore, there does not exist any classical solution and we can consider only nonclassical solutions of the form~\eqref{eq:nonclassicalsol}.
         \begin{enumerate}[\textbf{(N4a)}]
     \item[\textbf{(N4a)}] By assumption and~\eqref{eq:nonclassicalsol3} $f(\rho_L) < p(\rho_L+) \le \bar p$ and therefore $\hat\rho(\bar p) < \rho_L$ and by~\textbf{np3} we have $\bar p = p(\rho_L-)$.
     \item[\textbf{(N4b)}] By assumption and~\eqref{eq:nonclassicalsol3} $f(\rho_L) > p(\rho_L-) \ge \bar p$ and therefore $\hat\rho(\bar p) > \rho_L$ and by~\textbf{np2} we have $\bar p = p(\rho_L+)$.
         \end{enumerate}
     \item[\textbf{(C5)}] In this case $\left[ (t,x) \mapsto \mathcal{R}[\rho_L, \rho_R](x/t) \right]$ performs a possible null rarefaction on the left of the constraint and it satisfies~\eqref{eq:constrianedRiemann2} because $f(\rho_R) \le p(\rho_L-)$ and $p(\xi(t)) \equiv p(\rho_L-)$ by~\textbf{bp1}. Assume that there exists a nonclassical solution of the form~\eqref{eq:nonclassicalsol}. Since by assumption and~\eqref{eq:nonclassicalsol3} $\bar p \ge p(\rho_L+) > f(\rho_L)$, we have $\hat\rho(\bar p) < \rho_L$ and by~\textbf{np3} $\bar p = p(\rho_L-)$, but in this case the nonclassical solution coincides with the classical one.
     \item[\textbf{(N5a)}] In this case $\left[ (t,x) \mapsto \mathcal{R}[\rho_L, \rho_R](x/t) \right]$ does not satisfy~\eqref{eq:constrianedRiemann2} because $f(\rho_R) > p(\rho_L-)$, see case~\textbf{(C5)}. Therefore, there does not exist any classical solution and we can consider only nonclassical solutions of the form~\eqref{eq:nonclassicalsol}.
    \begin{enumerate}[\textbf{(N5b)}]
    \item[\textbf{(N5b)}] By assumption and~\eqref{eq:nonclassicalsol3}, $f(\rho_L) < p(\rho_L+)\le \bar p$ and therefore $\hat\rho(\bar p) < \rho_L$ and by~\textbf{np3} we have $\bar p = p(\rho_L-)$.
     \item[\textbf{(N5b)}] By assumption and~\eqref{eq:nonclassicalsol3}, $f(\rho_L) > p(\rho_L-) \ge \bar p$ and therefore $\hat\rho(\bar p) < \rho_L$ and by~\textbf{np2} we have $\bar p = p(\rho_L+)$.
    \end{enumerate}
    \end{enumerate}

\subsection{Proof of Proposition~\ref{prop:Riemann}}\label{sec:technicalRiemann}

\begin{enumerate}[\textbf{(R1)}]
\item[\textbf{(R1)}] Any solution given by $\mathcal{R}^\star$ coincides on each side of the constraint with a solution given by the classical Riemann solver $\mathcal{R}$. Therefore it satisfies the Rankine--Hugoniot jump condition along any of its discontinuities away from the constraint. Finally, by definition of $\hat\rho$ and $\check\rho$, it satisfies the Rankine--Hugoniot jump condition also along the constraint.
\item[\textbf{(R2)}] It is clear by the proof of Proposition~\ref{prop:riemann}.
\item[\textbf{(R3)}] It is proved as in~\textbf{(R1)} since any classical solution is in $\BV$.
\item[\textbf{(R4)}] As was proved in  Ref.~\cite{ColomboGoatinConstraint}, $\mathcal{R}^\star $ is continuous on $\mathcal{C}\cup\mathcal{N}$. If $(\rho_L,\rho_R)$ is not in $\mathcal{C}\cup\mathcal{N}$ then $ p$ experiences a jump at $\xi=\rho_L$. Therefore, the local in time solutions of the Riemann problem for the initial conditions $(\rho_L+\varepsilon,\rho_R)$ and $(\rho_L-\varepsilon,\rho_R)$ are different and only one of the two converges to  $\mathcal{R}^\star[\rho_L, \rho_R]$ as $\varepsilon>0$ goes to zero.
\item[\textbf{(R5)}] We first stress once again that we can discuss the consistency property of our Riemann solvers only locally in time because, in general, the solutions may be not even self--similar globally in time. However, locally in time, the efficiency of the exit can be assumed to be constant and it is thus sufficient to proceed as in Ref.~\cite{ColomboGoatinConstraint}.
\item[\textbf{(R6)}] It is clear by the proof of Proposition~\ref{prop:riemann}.
\end{enumerate}

\section{Further discussion on the model, conclusions and perspectives}\label{sec:conlusion}

The present model does not take into account extremal cases. For instance, whenever a high density is approaching the exit the efficiency of the exit can become very small. As a consequence, even a small density of pedestrians may form a queue provided a sufficiently high density is approaching from behind. However, at least in this case, the low efficiency of the exit has not a ``big'' effect on the flow at the exit which is in fact ``small''. Further investigation and modeling may be needed in order to deal with such singular effects.

The planned forthcoming papers of the authors aim to generalize the present model to the initial--boundary value problem with non--local constraint, to code the resulting model, to simulate realistic evacuations, to state and solve optimal management problems, to reproduce the so--called Braess' paradox for pedestrian flows (e.g.~the situation where a wider door placed before the exit door makes the evacuation faster), see Refs.~\cite{colombo2010macroscopic}, \cite{gakoto}, and to introduce further features of ``panic'' behavior, see Refs.~\cite{chalonsgoatinseguin}, \cite{ColomboRosini1}.

\section*{Acknowledgment}

The authors are partially supported by the French ANR JCJC grant CoToCoLa. The third author was partially supported by ICM, University of Warsaw, Narodowe Centrum Nauki, grant 4140, Polonium 2011 (French-Polish cooperation program) under the project ``CROwd Motion Modeling and Management'' and the Organizing Committee of HYP2012.

\bibliographystyle{acm}
    \bibliography{bibliography}

\end{document}